\documentclass[12pt,a4paper]{article}

\usepackage{amsmath,amstext,amssymb,amscd,euscript}
 \usepackage{graphicx}

\usepackage[russian,english]{babel}
\usepackage[cp1251]{inputenc}

\newcommand{\Xcomment}[1]{}
\newtheorem{theorem}{Theorem}[section]
\newtheorem{lemma}[theorem]{Lemma}
\newtheorem{corollary}[theorem]{Corollary}
\newtheorem{prop}[theorem]{Proposition}

\newenvironment{proof}{\noindent{\bf Proof}\/}%
{\hfill$\qed$\medskip}

\def\qed{\Box}

\makeatletter \@addtoreset{equation}{section} \makeatother

\newenvironment{numitem1}{\refstepcounter{equation}\begin{enumerate}%
\item[(\thesection.\arabic{equation})]}{\end{enumerate}}

\newcommand{\refeq}[1]{(\ref{eq:#1})}  

 \makeatletter
\renewcommand{\section}{\@startsection{section}{1}{0pt}%
{-3.5ex plus -1ex minus -.2ex}{2.3ex plus .2ex}%
{\normalfont\Large}}
 \makeatother

 \makeatletter
\renewcommand{\subsection}{\@startsection{subsection}{2}{0pt}%
{-3.0ex plus -1ex minus -.2ex}{1.5ex plus .2ex}%
{\normalfont\normalsize\bf}}
 \makeatother

 \newcommand{\SEC}[1]{\ref{sec:#1}}  

\def\Rset{{\mathbb R}}

\def\Zset{{\mathbb Z}}
\def\Qset{{\mathbb Q}}

\def\Ascr{{\cal A}}
\def\Bscr{{\cal B}}

\def\Escr{{\cal E}}

\def\Hscr{{\cal H}}
\def\Iscr{{\cal I}}
\def\Jscr{{\cal J}}

\def\Lscr{{\cal L}}

\def\Pscr{{\cal P}}
\def\Qscr{{\cal Q}}
\def\Rscr{{\cal R}}
\def\Sscr{{\cal S}}
\def\Tscr{{\cal T}}

\def\frakC{\mathfrak{C}}

\def\tilde{\widetilde}
\def\hat{\widehat}
\def\bar{\overline}
\def\eps{\varepsilon}
\def\orar{\overrightarrow}

\def\deltain{\delta^{\rm in}}
\def\deltaout{\delta^{\rm out}}

\def\convh{{\rm conv}}

\def\xmin{x^{\rm min}}
\def\xmax{x^{\rm max}}
\def\ymin{y^{\rm min}}
\def\ymax{y^{\rm max}}

\def\Lrot{\Lscr^{\rm rot}}

\def\Tscrlow{\Tscr^{\rm low}}
\def\Tscrup{\Tscr^{\rm up}}

\def\Pibeg{\Pi^{\rm beg}}
\def\Piend{\Pi^{\rm end}}
\def\Ebeg{E^{\rm beg}}
\def\Eend{E^{\rm end}}

\def\convex{{\rm conv}}
\def\dimen{{\rm dim}}

\def\rest#1{_{\,\vrule height 1.5ex width 0.05em depth 1pt\, #1}}

\def\pic{\pi^{\rm c}}

\begin{document}
\baselineskip=14pt
\parskip=3pt

\title{On stable assignments generated by choice functions of mixed type}

\author{Alexander V.~Karzanov
\thanks{Central Institute of Economics and Mathematics of
the RAS, 47, Nakhimovskii Prospect, 117418 Moscow, Russia; email:
akarzanov7@gmail.com.}
}
\date{}

 \maketitle
\vspace{-0.7cm}
 \begin{abstract}
We consider one variant of stable assignment problems in a bipartite graph endowed with nonnegative capacities on the edges and quotas on the vertices. It can be viewed as a generalization of the stable allocation problem introduced by Ba\H{\i}ou and Balinsky, which arises when strong linear orders of preferences on the vertices in the latter are replaced by weak ones. At the same time, our stability problem can be stated in the framework of a theory by Alkan and Gale on stable schedule matchings generated by choice functions of a wide scope. In our case, the choice functions are of a special, so-called \emph{mixed}, type. 

The main content of this paper is devoted to a study of \emph{rotations} in our mixed model, functions on the edges determining ``elementary'' transformations between close stable assignments. These look more sophisticated compared with rotations in the stable allocation problem (which are generated by simple cycles). We efficiently construct a poset of rotations and show that the stable assignments are in bijection with the so-called closed functions for this poset; this gives rise to a ``compact'' affine representation for the lattice of stable assignments and leads to an efficient method to find a stable assignment of minimum cost.
 \medskip
 
\noindent\emph{Keywords}: stable marriage, allocation, diversification, choice function, rotation
 \end{abstract}


\section{Introduction}  \label{sec:intr}

Since the famous work of Gale and Shapley~\cite{GS} on stable matchings (``marriages'') in bipartite graphs, there have appeared a variety of researches on stability models in graphs. A well-known model among them was introduced by Ba\H{\i}ou and Balinsky~\cite{BB} under the name of \emph{stable allocation problem}. 

In a setting of that problem, there is given a bipartite graph $G=(V,E)$ whose vertices are partitioned into two subsets (color classes), say, $F$ and $W$, interpreted as the sets of \emph{firms} and \emph{workers}. An edge $e\in E$ connecting vertices (or ``agents'') $u$ and $v$ represents a possible ``contract'' between them, and its intensity (or an \emph{assignment} on $e$) is restricted by a nonnegative real upper bound, or \emph{capacity}, $b(e)\in \Rset_+$. Besides, each vertex $v\in V$ is endowed with a quota $q(v)\in\Rset_+$ which restricts the total assignment on the set $E_v$ of edges incident to $v$, and there is a (strong) linear order $>_v$ establishing preferences on the edges in $E_v$, namely, $e>_v e'$ means that ``agent'' $v$ prefers ``contract'' $e$ to ``contract'' $e'$. An assignment (or \emph{allocation}) $x:E\to\Rset_+$ satisfying the capacity and quota constraints is called \emph{stable} if, roughly, there is no edge $e=\{u,v\}$ such that $x(e)$ can be increased, without violation of the capacities and quotas anywhere, possibly simultaneously decreasing $x$ on an edge $e'\in E_u$ with $e'<_u e$ and/or an edge $e''\in E_v$ with $e''<_v e$. 

Ba\H{\i}ou and Balinsky~\cite{BB} proved that a stable allocation always exists and can be found by an efficient algorithm. Also they showed that if the capacities and quotas are integer-valued, then there exists an integral stable allocation, and obtained some other impressive results. (Hereinafter an algorithm dealing with $G=(V,E)$ is said to be \emph{efficient} if it runs in time polynomial in $|V|,|E|$.)

Alkan and Gale~\cite{AG} considered much more general problem in a bipartite graph with capacities $b$ and, possibly, quotas $q$ as above. Here each vertex $v\in V$ is endowed with a \emph{choice function} $C_v$ that maps the set $\Bscr_v:=\{z\in \Rset_+^{E_v}\colon z(e)\le b(e)\; \forall e\in E_v\}$ into itself. The role of $C_v$ is twofold; in the case with quotas, this can be briefly explained as follows. First, when the amount $|z|:=\sum(z(e)\colon e\in E_v)$ exceeds the quota $q(v)$, it reduces $z$ so as to get the equality $|z'|=q(v)$, where $z'=C_v(z)$. Second, it gives a way to compare the vectors in $\Bscr_v$ obeying the quotas, by use of the preference relation $\succeq\,=\,\succeq_v$ defined by $z\succeq z'\Longleftrightarrow C_v(z\vee z')=z$ (where $z\vee z'$ takes the values $\max\{z(e),z'(e)\}$). By imposing to the choice functions $C_v$, $v\in V$, the axioms of \emph{consistence} and \emph{persistence} (or \emph{substitutability}) (which have been known, sometimes under different names, for simpler models, going back to Kelso and Crawford~\cite{KC}, Roth~\cite{roth}, Blair~\cite{blair}, and some other works) and by adding to them the axiom of \emph{size-monotonicity} (cf. Fleiner~\cite{flein}) or, stronger, of \emph{quota-filling}, Alkan and Gale developed a rather powerful and elegant theory of stability for weighted bipartite graphs. A central result in it (generalizing earlier results for particular cases) is that the set $\Sscr$ of stable assignments (or stable schedule matchings, in terminology of~\cite{AG}) is nonempty and forms a distributive lattice under the preference relations $\succeq_F$ determined by the choice functions $C_v$ with $v$ contained in the vertex part $F$ of $G$. (Here $x,y\in \Sscr$ satisfy $x\succeq_F y$ if and only if for each $f\in F$, the restrictions $z,z'$ of $x,y$ to $E_v$ (respectively) satisfy $z\succeq_f z'$.)

As an illustration of the setting with quotas, Alkan and Gale pointed out two special cases (Examples 1,2 in Sect.~2 of~\cite{AG}). One of them is just the above-mentioned stable allocation problem, that we will briefly denote as \emph{SAP}. In this case, for each $v\in V$, the choice function $C_v$ acts on $z\in\Bscr_v$ with $|z|> q(v)$ so as to respect more preferred and ignore less preferred edges in $E_v$. More precisely, $C_v$ selects the element $\hat e\in E_v$ such that $p:=\sum(z(e)\colon e>_v \hat e)< q(v)$ and $p+z(\hat e)\ge q(v)$, preserves the values for all $e>_v \hat e$, reduces to zero the values for all $e<_v \hat e$, and changes the value on $\hat e$ to $q(v)-p$. The second special case, that we abbreviate as \emph{SDP}, deals with the so-called \emph{diversifying} choice function for each $v\in V$, which uniformly reduces values of $z\in \Bscr_v$ with $|z|>q(v)$; more precisely, it computes the number (``height'') $r$ such that $\sum(\min\{z(e), r\}\colon e\in E_v)=q(v)$, and changes $z(e)$ to $\min\{z(e),r\}$, $e\in E_v$. In fact, this corresponds to the equivalence (or disorder) relation between all elements of $E_v$ from the viewpoint of the ``agent'' $v$, in contrast to the allocation case where $E_v$ is totally ordered. A rather surprising fact shown in~\cite{karz2} is that SDP is poor: for any $b$ and $q$, there exists \emph{exactly one} stable solution.

In this paper, we consider a common generalization of SAP and SDP, which is interesting theoretically and may find reasonable applications. It differs from SAP by admitting that for each $v\in V$, the preferences on $E_v$ are established by a \emph{weak} linear order. Equivalently, the set $E_v$ is partitioned into subsets $\pi^1,\ldots,\pi^k$, called \emph{ties}, such that $v$ prefers $e\in\pi^i$ to $e'\in\pi^j$ if $i<j$, and regards $e,e'$ as equivalent if $i=j$. We write $e>_v e'$ if $i<j$, and $e\sim_v e'$ if $i=j$. Then an assignment $x: E\to\Rset_+$ satisfying the capacity and quota constraints as before is stable if, roughly, there is no edge $e=\{u,v\}$ such that $x(e)$ can be increased, without violation of the capacities and quotas anywhere, possibly simultaneously decreasing $x$ on an edge $e'\in E_u$ such that either $e'<_u e$ and $x(e')>0$, or $e'\sim_u e$ and $x'(e)>x(e)$, and/or an edge $e''\in E_v$ such that either $e''<_v e$ and $x(e'')>0$, or $e''\sim_v e$ and $x''(e)>x(e)$. 
Our case is shown to be within the framework of Alkan-Gale's model as well; here for $v\in V$, the choice function $C_v$ respects the assignments on more preferred ties, ignore less preferred ties, and acts as in the diversifying case for one tie in between those (a precise definition will be given later). We refer to a $C_v$ of this sort as a choice function of \emph{mixed type}, and denote the arising stable assignment problem as \emph{SMP}. 

(It maybe of interest to note that the behavior of ``agents'' $v\in V$ in SMP can be expressed, to some extent, in terms of maximizing a certain concave utility function $U(z)$ over $z\in \Rset_+^{E_v}$ (subject to $z\le b\rest{E_v}$ and $|z|\le q(v)$). Considering the ties $\pi^1,\ldots,\pi^k$ for $\ge_v$ as above, one can take as $U(z)$ the function $\sum_{i=1}^k (k+1-i)z(\pi^i))- \eps\sum(z^2(e)\colon e\in E_v)$ for a sufficiently small $\eps>0$ (where $z(\pi^i):=\sum(z(e)\colon e\in\pi^i)$). Here the linear part of $U(z)$ reflects the preferences on the ties, while the quadratic part serves to ``align'' assignments within each tie.)

We will take advantages from both ways to state SMP: via weak linear orders (a \emph{combinatorial setting}), and via choice functions of mixed type (an \emph{analytic setting}). On the second way, relying on Alkan-Gale's theory in~\cite{AG}, we utilize the fact that the set $\Sscr=\Sscr^{\rm mix}(G,b,q,\ge)$ of stable assignments in SMP is nonempty and forms a distributive lattice under the corresponding preference relations on $\Sscr$, and appeal to the construction of the $F$-optimal stable assignment, that we denote as $\xmin$.

The main content of this paper is devoted to a study of \emph{rotations}, ``elementary augmentations'' connecting close stable assignments in SMP, which leads to clarifying the structure of the lattice on $\Sscr$. This is already beyond the content of~\cite{AG}, but is parallel, in a sense, to a study of rotations for SAP in~\cite{DM}. (Note that originally nice properties of rotations and their applications have been revealed in connection with the stable marriage and roommate problems in~\cite{irv,IL,ILG}. One can also mention a recent work~\cite{FZ} on the Boolean case of Alkan-Gale's model (generalizing stable matchings of many-to-many type), where rotations are extensively studied.) 

It turns out that the formation and properties of rotations in SMP are more intricate compared with SAP. Whereas a rotation in SAP is generated by a simple cycle whose edges have weights $\pm 1$, a rotation in SMP is represented by an integer-valued function $\rho$ on the edges of a certain strongly connected graph $L$, and it is possible that $\rho$ takes large values and $L$ has large vertex degrees. As a consequence, SMP with integral $b,q$ need not have an integral stable solution (in contrast to SAP); moreover, the convex hull of $\Sscr$ may have vertices with exponentially large denominators. 

Nevertheless, we can explicitly construct, in time polynomial in $|E|$, the set of all possible rotations $\rho$ in SMP, along with their maximal possible weights $\tau(\rho)$ (that can be used in ``elementary'' transformations $x\rightsquigarrow x+\tau(\rho)\rho$ of stable assignments). Moreover, we efficiently construct a poset of rotations $\rho$ endowed with weights $\tau(\rho)$ and show that the so-called \emph{closed functions} related to this poset one-to-one correspond to the stable assignments in SMP, yielding a sort of ``explicit and compact representation''  for $\Sscr$.

We then take advantages from this bijection between the set $\Sscr$ of stable assignments in SMP and the set $\frakC$ of closed functions for the rotational poset; this is in parallel with the corresponding results on the above-mentioned stability problems in~\cite{DM,FZ} and earlier results on rotations and their applications in stable matching problems of one-to-one and many-to-many types in~\cite{IL,ILG,BAM}. Using the fact that the number of rotations is $O(|E|)$, one can solve, in strongly polynomial time, the weighted strengthening of stable assignment problem: given weights (or \emph{costs}) $c(e)\in\Zset$ of edges $e\in E$, find $x\in\Sscr$ minimizing $cx:=\sum(c(e)x(e)\colon e\in E)$  (similar to results in~\cite{ILG,BAM,DM}). This is due to a reduction to the problem of minimizing a linear function over $\frakC$, which in turn is reduced to the classical minimum cut problem, by Picard's method~\cite{pic}.

As an additional observation, note that once we are able to minimize a linear function over $\Sscr$ in polynomial time, the separation problem over $\Sscr$ is polynomially solvable as well, due to the polynomial equivalence between linear optimization and separation problems established in~\cite{GLS} based on the ellipsoid method. 

(Recall that the \emph{separation problem} for $\Sscr$ is: given $\xi\in\Rset^E$, decide whether $\xi$ belongs to the convex hull $\convh(\Sscr)$ of $\Sscr$, and if not, find a hyperplane in $\Rset^E$ separating $\xi$ from $\Sscr$. Note that since the separation problem for $\Sscr$ is efficiently solvable, it is tempting to ask of a possibility to explicitly characterize facets  of the convex hull $\convh(\Sscr)$ of $\Sscr$. A characterization of this sort is well known for the set of stable marriages in a bipartite graph, due to Vande Vate~\cite{VV}. Another example of availability of a good facet characterization concerns the Boolean case of Alkan--Gale's model for which~\cite{FZ} shows that the convex hull of stable objects is affinely congruent to an appropriate Stanley's order polytope. One more example is the unconstrained case $b\equiv\infty$ of SAP for which an affine congruence to an order polytope is shown in~\cite{MS}. As to our SMP model, the facet characterization issue is open even in the unconstrained case.)

Concerning computational complexity aspects, it should be noted that all algorithms that we present in this paper are, in fact, strongly polynomial, i.e. (roughly) they use standard arithmetic and logic operations, and the number of these is bounded by a polynomial in the \emph{amount} (not the size) of input data; in a sense, they are adjusted to work with real $b,q$ as well.  Moreover, when dealing with integer-valued $b,q$, the memory size needed for storing intermediate data should be bounded by a polynomial in the size of input data. Note that we often do not care of precisely estimating time bounds of the algorithms that we devise and restrict ourselves merely by establishing  their (strongly) polynomial-time complexity; equivalently, we may say that an algorithm takes time polynomial in $|V|,|E|$.
\smallskip

This paper is organized as follows. Section~\SEC{def} contains basic definitions and settings and reviews important results from~\cite{AG} needed to us. Here we recall that in a general case: the set $\Sscr$ of stable assignments is nonempty; the partial orders $\succeq_F$ and $\succeq_W$ on $\Sscr$ that express preferences for the parts $F$ (``firms'') and $W$ (``workers'') are polar to each other, i.e., for $x,y\in\Sscr$, if $x\succeq_F y$ then $y\succeq_W x$, and vice versa; and $(\Sscr,\succeq_F)$ is a distributive lattice. In this lattice, we denote the best assignment for $F$ (and the worst for $W$) by $\xmin$, and the worst assignment for $F$ by $\xmax$. 
In Section~\SEC{begin} we recall the general method from~\cite{AG} to construct a sequence of assignments that converges to $\xmin$. This method may take infinitely many iterations even for our SMP model. For this reason, in the Appendix we devise an efficient modification to find $\xmin$ in SMP.

Section~\SEC{rotat} is devoted to an in-depth study of rotations and related aspects in SMP. Here for the lattice $(\Sscr,\succeq_F)$ and an element $x\in \Sscr-\{\xmax\}$, we consider the set $X^\eps(x)$ of stable assignments $x'$ satisfying $x\succ_F x'$ and $|x-x'|:=\sum(|x(e)-x'(e)|\colon e\in E)=\eps$ for a sufficiently small real $\eps>0$. The set $X^\eps(x)$ forms a polytope whose vertices can be characterized combinatorially. More precisely, we construct a certain directed graph $\orar\Gamma(x)$ such that the maximal strong (viz. strongly connected) components in it are in bijection with the vertices of $X^\eps(x)$. (Here a strong component $L$ of $\orar\Gamma(x)$ is called \emph{maximal} if no other strong component of $\orar\Gamma(x)$ is reachable by a directed path from $L$.) We denote by $\Lrot(x)$ the set of maximal strong components in $\orar\Gamma(x)$ and associate with each $L=(V_L,E_L)\in\Lrot$ a certain integer vector $\rho_L\in \Zset^E$ having the support $E_L$ and forming an integer circulation. 

This vector $\rho_L$ is just what we mean under the \emph{rotation} for $x$ generated by $L\in\Lrot(x)$. One shows that the vectors $x+\lambda\rho_L$, where $L\in\Lrot(x)$ and $\lambda=\eps/|\rho_L|$, are exactly the vertices of $X^\eps(x)$ (cf. Corollary~\ref{cor:vertP}). For each $L\in\Lrot(x)$, we define $\tau(\rho_L)$ to be the maximal value $\lambda\in\Rset_+$ such that $x+\lambda\rho_L$ is stable. For $0<\lambda\le\tau(\rho_L)$, we say that the assignment $x':=x+\lambda\rho_L$ (which is shown to be stable) is obtained from $x$ by the \emph{rotational shift} of weight $\lambda$ along $\rho_L$; when $\lambda=\tau(\rho_L)$, the shift is called \emph{full}. 

Section~\SEC{addit_prop} discusses additional, including computational, aspects of rotations. We explain that for $x\in\Sscr-\{\xmax\}$, the set of rotations $\rho=\rho_L$ applicable to $x$ (i.e. $L\in\Lrot$) and their maximal weights $\tau(\rho)$ can be found in time polynomial in $|E|$ (by solving a certain linear system). Also one shows that any $x\in\Sscr$ can be reached from $\xmin$ by a sequence of $O(|E|)$ rotational shifts. The section finishes with an example of rotation $\rho_L$ taking ``big'' values while $L$ is ``small''; this is extended to a series of rotations whose norms grow exponentially in the number of vertices (in Remark~3). 

Section~\SEC{poset_rot} presents crucial results of this paper (generalizing known properties for SAP and earlier stability problems). We first consider an arbitrary sequence $T$ of stable assignments $x_0,x_1,\ldots,x_N$ such that $x_0=\xmin$, $x_N=\xmax$, and each $x_i$ is obtained from $x_{i-1}$ by a full rotational shift, referring to $T$ as a \emph{route}, and show that the set $\Rscr$ of rotations $\rho$ and their maximal weights $\tau(\rho)$ applied to form a route are the same for all routes (cf. Proposition~\ref{pr:equal_rot}). This set is extended to poset $(\Rscr,\tau,\lessdot)$ in which $\rho,\rho'\in\Rscr$ obey relation $\rho\lessdot\rho'$ if $\rho$ occurs earlier than $\rho'$ in all routes.  One proves that the set $\Sscr$ of stable assignments in SMP is in bijection with the set $\frakC$ of closed functions for $(\Rscr,\tau,\lessdot)$  (cf. Proposition~\ref{pr:biject}). Here $\lambda:\Rscr\to\Rset_+$ with $\lambda\le\tau$ is called \emph{closed} if $\rho\lessdot\rho'$ and $\lambda(\rho')>0$ imply $\lambda(\rho)=\tau(\rho)$. This gives an infinite analog of Birkhoff's representation theorem, as we notice in Remark~4. We further explain that the relation $\lessdot$ can be computed efficiently, by constructing a directed graph $(\Rscr,\Escr)$ where the reachability by directed paths determines $\lessdot$. 

Based on these results, in Section~\SEC{weight}, we discuss an explicit affine representation $x=\xmin+A\lambda$ (where $A$ is an $E\times \Rscr$ matrix) involving stable assignments $x$ and closed functions $\lambda$, and finish with the minimum cost strengthening of SMP. Also we briefly discuss (in Remark~5) interrelations between optimization, separation and facet description issues concerning SMP and earlier stability problems.

In the Appendix, we describe an efficient algorithm to find the $F$-optimal stable assignment $\xmin$ in SMP. It modifies the method from Sect.~\SEC{begin}, by aggregating ``long'' sequences of iterations in the latter; this is based on solving a certain linear program. Also, as a partial alternative, we present (in Remark~6) an efficient and relatively simple ``combinatorial'' algorithm which, given an instance of SMP, recognizes the situation when a stable assignment must attain quotas with equalities in all vertices, and if so, finds an optimal stable assignment. The Appendix finishes with an open question.
 \vspace{-0.3cm}


\section{Definitions, settings, backgrounds}  \label{sec:def}

$\bullet$ ~The stability problem SMP that we deal with is defined as follows. There is given a bipartite graph $G=(V,E)$ in which the set $V$ of vertices is partitioned into two subsets (color classes) $F$ and $W$ whose elements are conditionally called \emph{firms} and \emph{workers}, respectively. The edges $e\in E$ are equipped with \emph{upper bounds}, or \emph{capacities}, $b(e)\in\Rset_+$, and the vertices $v\in V$ with  \emph{quotes} $q(v)\in\Rset_+$. We may assume that all values $b(e)$ and $q(v)$ are positive; otherwise the edge/vertex can be removed without affecting the problem. The set of edges incident to a vertex  $v\in V$ is denoted by $E_v$. 

For each vertex $v\in V$, there are \emph{preferences} on the edges in $E_v$; they are described by a \emph{weak} linear order, a transitive relation $\geq_v$ on pairs in $E_v$ in which different edges $e,e'$ are admitted to satisfy simultaneously both comparisons $e\geq_v e'$ and  $e'\geq_v e$. We write $e>_v e'$ if $e\geq_v e'$ but not $e'\geq_v e$, which means that from the viewpoint of the vertex (``agent'')  $v$, the edge (``contract'') $e$ is  \emph{preferred} to, or  \emph{better} than, the edge  $e'$;  whereas if both  $e\geq_v e'$ and  $e'\geq_v e$ take place, then the edges $e,e'$ are regarded by $v$ as \emph{equivalent}, which is denoted as $e\sim_v e'$. 

Equivalently, the weak order $\geq_v$ can be described via the ordered partition  $\Pi_v=(\pi^1,\pi^2,\ldots,\pi^{k})$ of the set $E_v$, where $e\in \pi^i$ and $e'\in\pi^j$ satisfy $e>_v e'$ if $i<j$, and $e\sim_v e'$ if $i=j$. We refer to $\Pi_v$ as a \emph{cortege}, and to its members $\pi^i$ as \emph{ties}, and may write $\pi^1>_v \pi^2>_v\cdots >_v\pi^k$. When $i<j$, we say that the tie $\pi^i$ is more preferable, or better, or earlier than $\pi^j$ (and $\pi^j$ is less preferable, or worse, or later than $\pi^i$).

Next we associate the weak order $\geq_v$ with an appropriate choice function $C_v$.
 \medskip

\noindent $\bullet$ ~The set  of \emph{admissible assignments} for $(G,b)$ (without involving quotas $q$) is defined to be the box $\Bscr:=\{x\in\Rset_+^E\colon x\le b\}$. For a vertex $v\in V$, the restriction of $x\in \Bscr$ to the edge set $E_v$ is denoted by $x_v$; the set of these restrictions is denoted as $\Bscr_v$. The \emph{choice function} (CF)  $C_v$ is a map of $\Bscr_v$ into itself such that: for each $z\in \Bscr_v$, there holds $C_v(z)\le z$ (as usual for choice functions), and $C_v$ is agreeable with the order $\geq_v$ and the quota $q(v)$ in the following way: 
  \begin{numitem1} \label{eq:z-q}
  \begin{itemize}
\item[(a)] if $|z|\le q(v)$ then $C_v(z)=z$;
\item[(b)] if $|z|\ge q(v)$, then in the cortege $\Pi_v$, we take the tie $\pi^i$ such that $z(\pi^1\cup\cdots\cup \pi^{i-1})< q(v)\le z(\pi^1\cup\cdots\cup \pi^{i})$ and define:
 \begin{equation} \label{eq:Cvz}
  C_v(z)(e):=\begin{cases} 
                   z(e) &\mbox{if $e\in \pi^j$, $j<i$,} \\
                        0 &\mbox{if $e\in \pi^j$, $j>i$,} \\
                        r\wedge z(e) &\mbox{if $e\in\pi^i$,}
                \end{cases}
   \end{equation}
where $r=r^z$ (\emph{the cutting height} in $\pi^i$) satisfies
  $$
  z(\pi^1\cup\cdots\cup \pi^{i-1})+\sum(r\wedge z(e) : e\in \pi^i)=q(v).
  $$
  \end{itemize}
  \end{numitem1}
(Hereinafter, for real-valued functions $a,b$ on a set $S$, we use the standard notation $a\wedge b$ and $a\vee b$ for the functions on $S$ taking values $\min\{a(e),b(e)\}$ and $\max\{a(e),b(e)\}$, $e\in S$, respectively. If $S'$ is a finite subset of $S$, we write $a(S')$ for $\sum(a(e) : e\in S')$. When $S$ is finite, $\sum(|a(e)|\colon e\in S)$ is denoted as $|a|$. So $a(S)=|a|$ if $a$ is nonnegative.)

The tie  $\pi^i$ figured in case~\refeq{z-q}(b) is called \emph{critical} and denoted as $\pic_v(z)$. The set of ties $\pi^i\in \Pi_v$ \emph{better} than $\pic_v(z)$ (i.e. $\pi^i>_v \pic_v(z)$) is denoted by $\Pibeg_v(z)$, and the set of ties $\pi^j$ \emph{worse} than $\pic_v(z)$ (i.e. $\pic_v(z)>\pi^j$)  by $\Piend_v(z)$). 
Also we denote the set of edges occurring in the ties in $\Pibeg_v(x)$ (resp. $\Piend_v(x)$) by $\Ebeg_v(x)$ (resp. $\Eend_v(x)$). Note that $C_v(z)$ takes zero values on all edges in $\Eend_v(z)$.
 
We call $C_v$ defined as in~\refeq{z-q} a choice function of  \emph{mixed type}, or simply a \emph{mixed} CF. Here the term ``mixed'' is justified by the fact that $C_v$ of this sort arises as a generalization of two known cases of choice functions, pointed out as examples in~\cite{AG}.
  \medskip
  
\noindent\textbf{Example~1.} Let $\geq_v$ be a strong linear order on $E_v$; in other words, each tie in the cortege $\Pi_v$ consists of a unique element. Then $C_v$ as above is a CF associated with the stable \emph{allocation} problem of Ba\H{\i}ou and Balinsky~\cite{BB} (denoted as SAP).
  \medskip
 
 \noindent\textbf{Example~2.} If all edges in $E_v$ are equivalent (and therefore $\Pi_v$ has a unique tie), then $C_v$ is called a \emph{diversifying} choice function, following terminology in~\cite{AG}.
  \medskip 

\noindent $\bullet$ ~Like the choice functions in the above examples, a CF of mixed type satisfies the standard axioms of consistence and persistence (in terminology of~\cite{AG}); let us check this for completeness of our description.

 \begin{lemma} \label{lm:axioms}
$C=C_v$ defined in~\refeq{z-q} satisfies the following conditions: for $z,z'\in\Bscr_v$,
 \begin{itemize}
\item[\rm(A1)] if $z\ge z'\ge C(z)$, then $C(z')=C(z)$ \emph{(consistence)};
\item[\rm(A2)] if $z\ge z'$, then $C(z)\wedge z'\leq C(z')$ \emph{(persistence}, or \emph{substitutability)}.
  \end{itemize}
  \end{lemma}
   \begin{proof}
~Let us check~(A1). In case $|z|\le q(v)$, we have $C(z)=z$, and $C(z')=C(z)$ trivially follows. And in case $|z|>q(v)$, consider the tie $\pi^i$ as in~\refeq{z-q}(b). Let $H:=\{e\in \pi^i\colon z(e)\ge r\}$. Then~\refeq{Cvz} shows that $C(z)(e)=z(e)$ for all $e\in\Ebeg_v(z) \cup(\pi^i-H)$;\; $C(z)(e)=r$ for all $e\in H$; and $C(z)(e)=0$ for all $e\in\Eend_v(z)$. Now $z\ge z'\ge C(z)$ implies $z'(e)=z(e)$ for all $e\in \Ebeg_v(z)\cup(\pi^i-H)$, and $z'(e)\ge r$ for all $e\in H$, whence $C(z')=C(z)$ easily follows.

Next we check~(A2). If $|z'|\le q(v)$, then $C(z')=z'\ge z'\wedge C(z)$, as required. Now let $|z'|> q(v)$. Then $z\ge z'$ implies $|z|> q(v)$. Define $\pi^i$ and $r$ as in~\refeq{z-q}(b) for $z$, and define $\pi^{i'}$ and $r'$ as in~\refeq{z-q}(b) for $z'$. Let $H:=\{e\in \pi^i\colon z(e)\ge r\}$ and  $H':=\{e\in\pi^{i'}\colon z'(e)\ge r'\}$. From $z\ge z'$ it follows that $i\le i'$. Suppose that $i<i'$. Then for $e\in \Ebeg_v(z)\cup(\pi^i-H)$, we have $C(z)(e)\wedge z'=z(e)\wedge z'(e)=z'(e)=C(z')(e)$, and for $e\in H$, we have $C(z)(e)\wedge z'(e)=r\wedge z'(e)\le C(z')(e)$. It follows that $C(z)\wedge z'\le C(z')$ (taking into account that $C(z)(e')=0$ for $e'\in\Eend_v(z)$). 

Finally, suppose that $i=i'$. Using $z\ge z'$ and $|C(z)|=|C(z')|=q(v)$, one can conclude that $r'\ge r$ and $H'\subseteq H$. As before, for $e\in \Ebeg_v(z)\cup(\pi^i-H)$, we have $C(z)(e)\wedge z'=z(e)\wedge z'(e)=C(z')(e)$. And for $e\in H$, we have $C(z)(e)\wedge z'(e)=r\wedge z'(e)\le r'\wedge z'(e)\le C(z')(e)$ (using $r'\ge r$ and $H'\subseteq H$). As a result, we obtain the desired relation $C(z)\wedge z'\le C(z')$.  
  \end{proof}

In particular, (A1) provides the equality $C(C(z))=C(z)$ for all $z\in \Bscr_v$. The set of assignments $z\in\Bscr_v$ such that $C(z)=z$ is denoted by $\Ascr_v$ (referring to $\Ascr_v$ as the set of \emph{stationary} assignments on $E_v$). Following a general model of Alkan and Gale, the set $\Ascr_v$ is endowed with the binary relation of ``revealed preferences'' $\succeq\,=\,\succeq_v$, where for $z,z'\in\Ascr_v$, one writes $z\succeq z'$ and says that $z$ is \emph{preferred to}, or \emph{better than}, $z'$ if
  \begin{equation} \label{eq:Czzp}
  C(z\vee z')=z.
  \end{equation}
Relation $\succeq$ is transitive and turns $\Ascr_v$ into a lattice, with join operation $\curlyvee$ viewed as $z\curlyvee z'=C(z\vee z')$, and meet operation $\curlywedge$ viewed as $z\curlywedge z'=C(\bar z\wedge \bar z')$, where for  $z''\in\Ascr_v$, the function (\emph{closure}) $\bar z''$ is defined to be $\sup\{ y\in\Bscr_v : C(y)=z''\}$.

One more axiom imposed on CFs in Alkan--Gale's theory is: 
 \begin{itemize}
\item[\rm(A3)] for $z,z'\in \Bscr_v$, if $z\ge z'$, then $|C_v(z)|\ge |C_v(z')|$ (\emph{size monotonicity}).
 \end{itemize}
A sharper form of~(A3) is the \emph{quota acceptability} axiom which says that $|C(z)|=\min\{|z|,q(v)\}$; this is automatically valid for the mixed CF $C_v$. 
  \medskip 

\noindent $\bullet$ ~For further needs, it is useful to specify the features of relation $\succeq$ in our case. We say that $z\in\Ascr_v$ is \emph{quota filling} if $|z|=q(v)$, and \emph{deficit} if $|z|<q(v)$. For a quota filling $z$, in the critical tie $\pic_v(z)=\pi^i$, we distinguish the subset 
  \begin{equation} \label{eq:HT}
  H_v(z):=\{e\in \pic_v(z) \colon  z(e)=r^z\},
  \end{equation}
called the \emph{head} of $z$. The set of edges occurring in $\Ebeg_v(z)\cup(\pic_v(z)-H_v(z))$ is denoted by $T_v(z)$ and called the \emph{tail} of $z$. When $z$ is deficit, we formally define $H_v(z):=\emptyset$ and $T_v(z):=E_v$. 

It is useful to reformulate the ``revealed preference'' relation for assignments in $\Ascr_v$ in terms of heads and tails, as follows.
  \begin{lemma} \label{lm:zsuccsp} 
{\rm(i)} For $z,z'\in\Ascr_v$,  the relation $z\succeq_v z'$ holds if and only if all edges $e\in T_v(z)$ satisfy $z(e)\ge z'(e)$.
 
{\rm(ii)} if  $z\succ_v z'$ and both $z,z'$ are quota filling, then $\pic_v(z)\ge\pic_v(z')$, and in case $\pic_v(z)=\pic_v(z')$, there hold $H_v(z')\subseteq H_v(z)$ and $r^{z'}>r^z$.
  \end{lemma}
 \begin{proof} ~The inequalities for edges in~(i) easily imply~\refeq{Czzp} if $|z|<q(v)$ (in view of $T_v(z)=E_v$). And if $|z|=q(v)$, then these inequalities imply $C(z\vee z')(e)=C(z)(e)=z(e)$ for all $e\in T_v(z)$. Also $|z|=q(v)$ implies $|C(z)|=q(v)=|C(z\vee z')|$. These relations give $C(z\vee z')(e')=C(z)(e')$ for all $e'\in H_v(z)$, and~\refeq{Czzp} follows. The converse implication in~(i) is easy as well.  

To see~(ii), let $\Pi_v=(\pi^1>\cdots>\pi^k)$ and $\pic_v(z)=\pi^i$. Using the inequalities $z(e)\ge z'(e)$ for $e\in T_v(z)$ (valid by~(i)), we have
  $$
z'(H_v(z))+z'(\Eend_v(z))=q(v)-z'(T_v(z)) 
 \ge q(v)-z(T_v(z))=z(H_v(z)).
 $$
It follows that $\pic_v(z')\in\{\pi^i,\ldots,\pi^k\}$, i.e. $\pic_v(z')\le\pi^i=\pic_v(z)$. If  $\pic_v(z')=\pic_v(z)$ takes place, then $z'(H_v(z))\ge z(H_v(z))=r^z|H_v(z)|$ (since $z'(\Eend_v(z))=0$). Then 
  $$
  r^{z'}\ge\max\{z'(e)\colon e\in H_v(z)\}\ge r^z.
  $$
This and $z'(e)\le z(e)< r^z$ for all $e\in \pic_v(z)-H_v(z)\subseteq T_v(z)$ imply $H_v(z')\subseteq H_v(z)$ and $r^{z'}>r^z$ (the equality $r^{z'}=r^z$ would lead to $z'=z$, contrary to $z'\prec_v z$.)
 \end{proof}

\noindent$\bullet$ ~The following definition is of importance.
 \medskip
 
\noindent\textbf{Definition 1.} For a vertex $v$ and an assignment $z\in\Ascr_v$, an edge $e\in E_v$ is called \emph{interesting} (in the sense that an increase at $e$ could cause the appearance of an assignment better than $z$) if there exists $z'\in\Bscr_v$ such that $z'(e)>z(e)$, $z'(e')=z(e')$ for all $e'\ne e$, and $C_v(z')(e)>z(e)$. (In terminology of~\cite{AG}, such an $e$ is called \emph{non-satiated}.)
 \medskip

For the mixed CF $C_v$, one can see that: 
 \begin{numitem1} \label{eq:int-mixed}
for $z\in\Ascr_v$, an edge $e\in E_v$ with $z(e)<b(e)$ is interesting for $v$ under $z$ if and only if $e$ belongs to the tail $T_v(z)$.
  \end{numitem1}
Indeed, let $z'$ be obtained from $z$ by increasing the value at $e$ by a small amount $\eps>0$. Then $z'\in \Bscr_v$. In case $|z|<q(v)$, we have $e\in T_v(z)$, $z'\in\Ascr_v$, and $e$ is interesting under $z$ (since $C(z')(e)=z'(e)>z(e)$). Now let $|z|=q(v)$. If $e\in T_v(z)$, then $C_v$ does not change $z'$ within the tail $T_v(z)$ and decreases the value on each edge in $H_v(z)$ by $\eps/|H_v(z)|$ (cf.~\refeq{Cvz}); therefore, $C_v(z')(e)=z'(e)>z(e)$, and $e$ is interesting under $z$. And if $e\notin T_v(z)$ (i.e., $e$ is in $H_v(z)$ or in $\Eend_v(z)$), then $C_v$ applied to $z'$ decreases $z'(e)$ by $\eps$. So $C(z')=z$ and $e$ is not interesting under $z$.
\medskip
 
\noindent $\bullet$ ~Now we leave ``individual agents'' and come to assignments on the whole edge set $E$. We define $\Ascr$ to be the set of admissible assignments $x\in\Bscr$ such that for all $v\in V$, the restriction $x_v$ of $x$ to $E_v$ belongs to $\Ascr_v$ (referring to $\Ascr$ as the set of \emph{fully stationary} assignments on $E$). So in our case, $\Ascr$ consists of the assignments respecting both the capacities $b$ and the quotas $q$.

In what follows, an edge connecting vertices $f\in F$ and $w\in W$ is denoted as $fw$.
\medskip

\noindent\textbf{Definition 2.} Let $x\in\Ascr$. An edge $e=fw\in E$ is called \emph{blocking} for $x$ if $e$ is interesting for both vertices $f$ and $w$ (when considering the restrictions $x_f=x\rest{E_f}$ on $\Ascr_f$, and $x_w$ on $\Ascr_w$). An assignment $x$ is called \emph{stable} if it admits no blocking edge.
 \medskip
 
We denote the set of stable assignments by $\Sscr=\Sscr(G,b,q,\ge)$. Using~\refeq{int-mixed}, we can characterize the blocking edges for assignments in $\Ascr$ as follows:

  \begin{numitem1} \label{eq:blocking}
for  $x\in\Ascr$, an edge $e=fw\in E$ is blocking if and only if $x(e)<b(e)$ and $e$ belongs to both tails $T_f(x)$ and $T_w(x)$ (where $T_v(x)$ means $T_v(x_v)$).
 \end{numitem1}

For stable $x,y\in\Sscr$, when the relation $x_f\succeq_f y_f$ takes place for all ``firms'' $f\in F$, we write $x\succeq_F y$; similarly, when $x_w\succeq_w y_w$ holds for all ``workers'' $w\in W$, we write $x\succeq_W y$. According to general results on stability for bipartite graphs shown by Alkan and Gale in~\cite{AG}, subject to axioms~(A1)--(A3), the following properties are valid:
  \begin{numitem1} \label{eq:stab-mix}
the set $\Sscr$ of stable assignments is nonempty, and the relation $\succeq_F$ turns $\Sscr$ into a distributive lattice (with operations $\curlywedge$ and $\curlyvee$ induced by those in the restrictions $\Ascr_f$ for $f\in F$); furthermore:
 \begin{itemize}
\item[(a)] (\emph{polarity}) the order $\succeq_F$ is converse to $\succeq_W$, namely: if $x\succeq_F y$, then $y\succeq_W x$, and vice versa;
\item[(b)] (\emph{size invariance}) for each fixed $v\in V$, the value $|x_v|$ is the same for all $x\in \Sscr$.
 \end{itemize}
 \end{numitem1}
We denote the minimal and maximal elements in $(\Sscr,\succeq_W)$ by $\xmin$ and $\xmax$, respectively (so the former is the best, and the latter is the worst w.r.t. the part $F$). 

Let us say that a vertex $v$ is \emph{fully filled} by $x\in\Sscr$ if the restriction function $x_v$ is quota filling (i.e. $|x_v|=q(v)$), and \emph{deficit} otherwise (i.e. $|x_v|<q(v)$); this does not depend on the choice of stable $x$, by~\refeq{stab-mix}(b). We will use one more property shown in~\cite[Corollary~3]{AG} for the quota-acceptable case (and therefore, applicable to our mixed model); it strengthens~\refeq{stab-mix}(b) for deficit vertices, saying that
  \begin{numitem1} \label{eq:strength}
if $v\in V$ is deficit, then the restriction functions $x_v$ are the same for all $x\in \Sscr$.
 \end{numitem1}


\section{Finding an initial stable assignment}  \label{sec:begin}

In order to construct the stable matching $\xmin$ (which is the best for the part $F$ and the worst for the part $W$) in our mixed problem SMP, one can use a method due to Alkan and Gale applied in~\cite[Sec.~3.1]{AG} to show the existence of a stable assignment in a general case considered there. It iteratively constructs triples of functions $(b^i,x^i,y^i)$ for $i=0,1,\ldots,i,\ldots$\,\,.

In the beginning of the process, one puts $b^0:=b$. A current, $i$-th, iteration considers the function $b^i$ (which is already known) and proceeds as follows. 

At the \emph{first stage} of the iteration, $b^i$ is transformed into $x^i\in\Rset_+^E$, by putting  $x^i_f:=C_f(b^i_f)$ for each $f\in F$. 
(Then $x^i$ satisfies the quotas $q(f)$ for all ``firms'' $f\in F$, when dealing with our setting with quotas; however, $x^i$ can exceed quotas $q(w)$ for some ``workers'' $w\in W$.)

At the \emph{second stage}, the obtained $x^i$ is transformed into $y^i\in\Rset_+^E$, by putting $y^i_w:= C_w(x^i_w)$ for each $w\in W$. (This $y^i$ already satisfies the quotas $q(v)$ for all $v\in V$, but it need not be stable.)

The obtained functions $x^i$ and $y^i$ are then used to generate the upper bounds $b^{i+1}$ for the next iteration, by putting
 \begin{equation} \label{eq:bi}
  b^{i+1}(e):=\begin{cases} 
                   b^i(e) &\mbox{if $y^i(e)=x^i(e)$,} \\
                   y^i(e) &\mbox{if $y^i(e)<x^i(e)$.}
                \end{cases}
   \end{equation}
Then $b^{i+1}$ is transformed into $x^{i+1}$ and, further, into $y^{i+1}$, and so on.

The process terminates if at some iteration, $i$-th say, we obtain the equality $y^i=x^i$ (which implies $b^{i+1}=b^i$, by~\refeq{bi}); then $x^i$ is just the required assignment (see the proposition below). 
In reality, it may happen that the sequence of iterations is infinite. Nevertheless, one can see (by arguing as in~\cite{AG}) that the process always converges. Indeed, the construction shows that for each $i$,
  \begin{equation} \label{eq:converg}
  b^{i}\ge x^i\ge y^i\qquad \mbox{and} \qquad  b^i\ge b^{i+1}\ge 0.
  \end{equation}
Therefore, $b^0,b^1,\ldots$ converge to some function $\hat b\ge 0$. This implies that $x^0,x^1,\ldots$ converge to the function $\hat x$, defined by $\hat x_f:=C_f(\hat b_f)$ for $f\in F$, and, further, that $y^0,y^1,\ldots$ converge to the function $\hat y$ which is equal to $\hat x$ (assuming that all choice functions $C_v$ are continuous). 

\begin{prop} \label{prop:limit} The limit function $\hat x$ is stable and, moreover,
$\hat x$ is optimal for $F$, i.e., $\hat x=\xmin$.
 \end{prop}

This result was proved in~\cite{AG} for a general case (subject to axioms (A1)--(A3)); see Theorem~1 (for the stability) and Theorem~2 (for the optimality) in that paper. 

Fortunately, for our mixed model the above method can be modified, by aggregating some series of consecutive iterations, so as to obtain an efficient algorithm to find the assignment $\xmin$; its running time is polynomial in $|E|$. We describe this modified method for SMP in the Appendix.


\section{Rotations}  \label{sec:rotat}

Throughout this section, we consider a fixed stable assignment $x\in\Sscr$ for our mixed problem SMP. Our aim is to characterize the set of stable assignments $x'$ that are close to $x$ and satisfy $x\succ_F x'$, namely, the set
  \begin{equation} \label{eq:Xeps}
  X^\eps=X^\eps(x):=\{x'\in\Sscr\colon x'\prec_F x,\; |x'-x|=\eps\}
  \end{equation}
for a sufficiently small real $\eps>0$ (where $|x'-x|$ denotes $\sum(|x'(e)-x(e)|\colon e\in E)$). We will use a technique of \emph{rotations}, generalizing a concept and constructions developed for earlier stability problems, starting with~\cite{IL}. (In particular, it will be shown that $X^\eps$ is a section of the polyhedral cone pointed at $x$ and generated by the rotations regarded as vectors in $\Rset^E$.) First of all we specify some notation needed in what follows. 
\smallskip

\noindent$\bullet$ We denote by $V^=$ the set of \emph{fully filled} vertices $v\in V$, which means that $|x_v|=q(v)$. The sets $V^=\cap F$ and $V^=\cap W$ are denoted by $F^=$ and $W^=$, respectively. (By~\refeq{stab-mix}(b), these sets do not depend on the choice of $x\in\Sscr$. In view of~\refeq{strength}, any transformation $x\rightsquigarrow x'\in\Sscr$ must involve only edges connecting $F^=$ and $W^=$.) The vertices in $V-V^=$ are called \emph{deficit}.
\smallskip

\noindent$\bullet$
Following definitions and notation from Sect.~\SEC{def} and considering $x$ as above, for a fully filled vertex $v\in V^=$, we denote by $\pic_v(x)$, $H_v(x)$, $\Pibeg_v(x)$, $\Piend_v(x)$, respectively, the \emph{critical tie} for $x$ in the cortege $\Pi_v$; the \emph{head} of $x$ in $\pic_v(x)$;  the set of ties in $\Pi_v$ better than $\pic_v$; the set of ties worse than $\pic_v$. Also we denote by $\Ebeg_v(x)$ and $\Eend_v(x)$ the sets of edges contained in the ties of $\Pibeg_v(x)$ and $\Piend_v(x)$, respectively. The complement to $H_v(x)$ in $\Ebeg_v(x)\cup\pic_v(x)$ is what we call the \emph{tail} for $x$ and $v$ and denote by $T_v(x)$.  ((Hereinafter, in brackets  we write $x$ instead of $x_v=x\rest{E_v}$ for short. Note also that since $|x_v|=q(v)$ and $\pic_v(x)$ is critical, we have: $x(\pic_v)>0$, $H_v(x)\ne\emptyset$, and $x$ is zero on all edges in $\Eend_v(x)$.) When $v$ is deficit, one sets $H_v(x):=\emptyset$ and $T_v(x):=E_v$.
\smallskip

\noindent$\bullet$
An edge $e\in E$ is called \emph{saturated} by $x$ if $x(e)=b(e)$, and \emph{unsaturated} otherwise.
\smallskip

\noindent$\bullet$ When otherwise is not explicitly said, we write $\succeq$ for $\succeq_F$.
 \smallskip

We start our description with the observation that property~\refeq{blocking} implies that the stability of $x$ is equivalent to the following property:
  \begin{numitem1} \label{eq:nonfill}
for each unsaturated edge $e=fw\in E$, at least one vertex $v$ among $f$ and $w$ is such that: $v$ is fully filled and $e$ belongs to $H_v(x)\cup \Eend_v(x)$. 
 \end{numitem1}

The construction of rotations is based on an analysis of stable assignments close to $x$. It is convenient to assume that $\eps$ is chosen so that:
  \begin{numitem1} \label{eq:epsE}
$\eps$ is less than any nonzero value among $x(e)$ and $b(e)-x(e)$ for $e\in E$, and less than $x(e)-x(e')$ for edges $e\in H_w(x)$ and $e'\in\pic_w(x)-H_w(x)$ among all $v\in W^=$.
 \end{numitem1}
 
Then useful relations between $x$ and $x'\in X^\eps$ are described in the next two lemmas. 
  \begin{lemma} \label{lm:xwyw}
Let $w\in W^=$ and $x'_w\ne x_w$. Then $\pic_w(x')=\pic_w(x)$, \;$H_w(x')=H_w(x)$, and $H_w(x)$ is exactly the set of edges $e\in E_w$ such that $x'(e)<x(e)$.
  \end{lemma}
 \begin{proof}  
~ By the polarity~\refeq{stab-mix}(b), $x'\prec_F x$ implies $x'_w\succ_w x_w$. Applying Lemma~\ref{lm:zsuccsp}(ii) to $z:=x'_w$ and $z':=x_w$, we have $\pic_w(x')\ge\pic_w(x)$. If $\pic_w(x')>\pic_w(x)$, then taking $e\in H_w(x)$, we obtain $\eps=|x-x'|\ge x(e)-x'(e)=x(e)>\eps$ (in view of~\refeq{epsE}); a contradiction. Therefore, $\pic_w(x')=\pic_w(x)$ and $H_w(x')\supseteq H_w(x)$. Moreover,   $H_w(x')=H_w(x)$ must take place (for otherwise, taking $e\in H_w(x)$ and $e'\in H_w(x')- H_w(x)$, we would have $\eps\ge(x(e)-x'(e))+(x'(e')-x(e')=x(e)-x(e')>\eps$). Finally, for $e\in E_w-H_w(x)$, we have $x'(e)\ge x(e)$ if $e\in T_w(x)$ (by Lemma~\ref{lm:zsuccsp}(i)), and $x'(e)=x(e)=0$ if $e\in\Eend_w(x)$, implying  $H_w(x)=\{e\in E_w\colon x'(e)<x(e)\}$.
\end{proof}

\noindent\textbf{Definition~3.} For $f\in F^=$, let $\hat \Pi_f(x)$ denote the set of ties $\pi^i\in \Pi_f$ such that: (a) the set $D^i=D^i_f(x)$ of unsaturated edges $e=fw\in \pi^i$ such that $e\in T_w(x)$ is nonempty, and (b) there is no tie better or equal to $\pi^i$ that contains an unsaturated edge incident to a deficit vertex in $W$. When $\hat \Pi_f(x)$ is nonempty, we denote the best (most preferred) tie $\pi^i$ in  $\hat \Pi_f(x)$ by $\hat \pi_f(x)$, and denote $D^i$ by $D_f(x)$, called the \emph{potential head} for $f$ and $x$. Otherwise we put $\hat \pi_f(x):=\{\emptyset\}$ and $D_f(x):=\emptyset$. When $\hat \pi_f(x)=\pic_f(x)$, we also denote $D_f(x)$ as $\tilde H_f(x)$ and call it the \emph{reduced head}. 
  \medskip
  
(Note that since $x$ is stable, each edge $e$ satisfying the condition in~(a) (if exists) belongs to $H_f(x)\cup \Eend_f(x)$, cf.~\refeq{nonfill}. Also if an edge $e=fw$ is unsaturated and $w$ is deficit, then $e$ cannot belong to $T_f(x)$. In case $\hat \pi_f(x)=\pic_f(x)$, we have $D_f(x)=\tilde H_f(x)\subseteq H_f(x)$ and any edge $fw$ in $H_f(x)-\tilde H_f(x)$ belongs to $T_w(x)$.)

 \begin{lemma} \label{lm:xfyf}
Let $f\in F^=$ and $x'_f\ne x_f$. Then $\pic_f(x')\le \pic_f(x)$,\; $H_f(x')$ coincides with the set of edges $\{e\in E_f\colon x'(e)>x(e)\}$, and if $\pic_f(x')=\pic_f(x)$ then $H_f(x')\subseteq H_f(x)$. Furthermore, $\pic_f(x')=\hat \pi_f(x)$ and $H_f(x')=D_f(x)$.
  \end{lemma}
  \begin{proof}
~The inequality $\pic_f(x')\le \pic_f(x)$ and the inclusion $H_f(x')\subseteq H_f(x)$ in case $\pic_f(x')=\pic_f(x)$ follow from Lemma~\ref{lm:zsuccsp}(ii) with $z:=x_f$ and $z':=x'_f$. In order to see the equality $H_f(x')=\{e\in E_f\colon x'(e)>x(e)\}$, observe that for an edge $e=fw$, if $x'(e)>x(e)$, then $x'(e)\le x(e)+\eps<b(e)$ (in view of~\refeq{epsE}) and there holds $e\in T_w(x)=T_w(x')$ (by Lemma~\ref{lm:xwyw}). Therefore, $e\notin T_f(x')$ (since $x'(e)<b(e)$ and $x'$ is stable, cp.~\refeq{nonfill}). This implies $e\in H_f(x')$ (taking into account that $e\in \Eend_f(x')$ is impossible since $x'(e)>0$).

Finally, by the stability of $x'$, there is no edge $e=fw$ in $\Ebeg_f(x')\cup (\pic_f-H_f(x'))$ such that $x'(e)<b(e)$ and $e\in T_w(x')$ ($=T_w(x)$) (in particular, $w$ cannot be deficit). Also $H_f(x')$ is nonempty and each edge $e$ in it is unsaturated by $x$ (since $x(e)<x'(e)$). Using these facts, one can conclude that $\pic_f(x')=\hat \pi_f(x)$ and $H_f(x')=D_f(x)$.
  \end{proof}

Lemmas~\ref{lm:xwyw} and~\ref{lm:xfyf} together with property~\refeq{strength} imply that
  \begin{numitem1} \label{eq:no_deficit}
for $x'\in X^\eps$, if $w\in W^=$ and $x'_w\ne x_w$, then no edge in $H_w(x)$ is incident to a deficit vertex in $F$; and if $f\in F^=$ and $x'_f\ne x_f$, then $D_f(x)\ne\emptyset$.
  \end{numitem1}

This can be generalized as follows.
  \begin{lemma} \label{lm:sing_path}
Let $(v_0,e_1,v_1,\ldots, e_k,v_k)$ be a path in $G$ such that: the vertices $v_0,\ldots,v_{k-1}$ are fully filled, and for $i=1,\ldots,k$, the edge $e_i$ connecting $v_{i-1}$ and $v_i$ belongs to $H_{v_{i-1}}(x)$ (when $v_{i-1}\in W^=$) or $D_{v_{i-1}}(x)$ (when $v_{i-1}\in F^=$). Suppose that either $v_k\in F-F^=$, or $v_k\in W$ and $D_{v_{k-1}}(x)=\emptyset$. Then any $x'\in X^\eps(x)$ and $i=1,\ldots,k$ satisfy $x'(e_i)=x(e_i)$.
  \end{lemma}

 \begin{proof}  
~This is shown by induction on $k-i$. If $v_k\in F$ is deficit, we have $x'(e_k)=x(e_k)$ (by~\refeq{strength}), implying $x'(e)=x(e)$ for all edges $e\in E_{v_{k-1}}$ (since $e_k\in H_w(x')=H_w(x)$). And if $v_k\in W$ and $D_{v_{k-1}}(x)=\emptyset$, then $x'_{v_{k-1}}=x_{v_{k-1}}$. See~\refeq{no_deficit}. Now for $i=0,\ldots,k-2$, the equality $x'(e_i)=x(e_i)$ follows from $x'(e_{i+1})=x(e_{i+1})$ and the balance condition $\sum(x'(a)-x(a)\colon a\in E_{v_i})=0$, taking into account that for $a,b\in E_{v_i}$, the values $x'(a)-x(a)$ and $x'(b)-x(b)$ are equal if both $a,b$ belong to the corresponding head ($H_{v_i}(x)$ or $D_{v_i}(x)$), and have the same sign if both $a,b$ are not in this head.
 \end{proof}
   
In what follows, since $x$ is fixed, we may abbreviate notation $H_v(x),  T_v(x), D_v(x)$  to  $H_v, T_v, D_v$. 
 
Lemma~\ref{lm:sing_path} motivates the following \emph{Cleaning procedure} which is intended to extract a set $V^0$ of vertices $v\in F^=\cup W^=$ where no changes of $x$ on $E_v$ are possible.
 \begin{itemize}
\item[(C):] We grow $V^0$, step by step, as follows. Initially, relying on~\refeq{no_deficit}, we insert in $V^0$ each vertex $w\in W^=$ whose head $H_v$ contains an edge incident to a deficit vertex, and each vertex $f\in F^=$ whose potential head $D_f$ is empty. At a current  step, if there appears a (new) vertex $v\in F^=\cup W^=$ whose head ($D_v$ or $H_v$) has an edge incident to a vertex already contained in $V^0$, then we insert $v$ in $V^0$ as well. 
 \end{itemize}

The fully filled vertices $v\in V^=$ occurring in $V^0$ upon termination of this procedure are called \emph{singular}, and the remaining fully filled vertices are called \emph{regular}. We denote $V^+:=V^=-V^0$, \,$F^+:=V^+\cap F$ and $W^+:=V^+\cap W$.    Note that
  \begin{numitem1} \label{eq:f-w}
for each $f\in F^+$, we have $D_f\ne\emptyset$, and all edges $fw\in D_f$ satisfy $w\in W^+$; similarly, for each $w\in W^+$, all edges $fw\in H_w$ satisfy $f\in F^+$.
 \end{numitem1}
 
By reasonings above, all transformations $x\rightsquigarrow x'\in X^\eps(x)$ can be performed only within the sets $H_w$ for regular $w\in W^+$ and $D_f$ for regular $f\in F^+$; we call the edges in these sets \emph{active}. (The heads $H_w$ and $D_f$ may be called active as well.)

(As we have seen, the construction of active sets for $F^+$ looks more involved compared with those for $W^+$. The reason of this asymmetry is caused by their different roles in a small transformation of $x$; namely, to provide more preferred assignments in the part $W$, and less preferred in the part $F$,  while preserving stability, we should slightly decrease $x$ within full current heads for $W$, but increase it either in reduced heads for $F$, or even in ties later than the critical ones where all values are zeros.)

Next we will associate with the regular set $V^+$ two important structures: a graph generated by active edges, and a system of linear equations. We first prefer to describe the second structure. 

By Lemmas~\ref{lm:xwyw} and~\ref{lm:xfyf}, the closeness of $x'$ to $x$ subject to $x\succ_F x'$  corresponds to the property of \emph{homogeneity}, namely:
  \begin{numitem1} \label{eq:homogen}
for each $f\in F^+$, the values $x'(e)-x(e)$ are equal to the same number $\phi_f\ge 0$ for all edges $e\in D_f$; and for each $w\in W^+$, the values  $x(e)-x'(e)$ are equal to the same number $\psi_w\ge 0$ for all $e\in H_w$.
  \end{numitem1}
  
These values $\phi_f$ and $\psi_w$ serve as variables in the following system where we use as coefficients the scalars $h_f:=|D_f|$ and $h_w:=|H_w|$: 
  \begin{gather}
   h_f\phi_f-\sum(\psi_w\colon w\in W^+,\; fw\in E_f-D_f)=0,\quad f\in F^+; \label{eq:lp1}\\
  h_w\psi_w-\sum(\phi_f\colon f\in F^+,\; fw\in E_w-H_w)=0,\quad w\in W^+; \label{eq:lp2}\\
   \sum(h_f\phi_f\colon f\in F^+)+\sum(h_w\psi_w\colon w\in W^+)= \eps;    \label{eq:lp3} \\
      \phi\ge 0,\; \psi\ge 0.  \label{eq:lp4}
   \end{gather}

Here $\phi_f$ means a uniform increase of $x(e)$ at the edges $e\in D_f$, and $\psi_w$ a uniform decrease of $x(e)$ at the edges $e\in H_w$. (Recall that for $x'\in X^\eps$ with $\eps$ small, $H_f(x')=D_f(x)$ for all $f\in F^+$, and $H_w(x')=H_w(x)$ for all $w\in W^+$.)  The equalities in~\refeq{lp1} describe the balance conditions for the vertices $f\in F^+$ (saying that the total increase of $x$ on $D_f$ is compensated by the total decrease on $E_f-D_f$, keeping the corresponding quota fillings), and the ones in~\refeq{lp2} describe similar balance conditions for the vertices $w\in W^+$. And~\refeq{lp3} expresses the condition $|x'-x)=\eps$.
   
Our aim is to show that the set (polytope)
  \begin{equation} \label{eq:polPscr}
  \Pscr^\eps=\Pscr^\eps(x):=\{(\phi,\psi)\}\subset \Rset^{F^+}\times \Rset^{W^+},
  \end{equation}
where $(\phi,\psi)$ runs through the solutions to system~\refeq{lp1}--\refeq{lp4}, matches the set of augmentations $x'-x$ for all $x' \in X^\eps(x)$. 

More precisely, let us associate with $(\phi,\psi)\in\Pscr^\eps$  the following vector $\Delta^{\phi,\psi}\in\Rset^E$:
  \begin{numitem1} \label{eq:Deltapp}
for $e=fw\in E$, define $\Delta^{\phi,\psi}(e):=\phi_f$ if $f\in F^+$ and $e\in D_f$,\; $-\psi_j$ if $w\in W^+$ and $e\in H_w$,\; and 0 otherwise.
  \end{numitem1} 
 \begin{theorem} \label{tm:close}
Let $x\in\Sscr-\{\xmax\}$. Then:

{\rm(a)} for a sufficiently small real $\eps>0$ (in particular, satisfying~\refeq{epsE}), the set $X^\eps(x)$ defined in~\refeq{Xeps} coincides with the polytope $\{ x+\Delta^{\phi,\psi}\colon (\phi,\psi)\in\Pscr^\eps\}$;
  
 {\rm(b)} if $y\in\Sscr$ and $y\prec x$, then for a sufficiently small $\eps>0$, there exists $x'\in X^\eps(x)$ satisfying $y\preceq x'\prec x$.
  \end{theorem}
  
\noindent (Also we shall see that system~\refeq{lp1}--\refeq{lp4} has no solution if and only if $x=\xmax$.)
  
The proof of this theorem is given in the rest of this section. One auxiliary assertion that we use in the proof, which is  interesting in its own right, is as follows.

 \begin{lemma} \label{lm:plus-minus}
Let $A$ and $B$ be two disjoint nonempty subsets of $E$, and let $(J,A\cup B)$ be the subgraph of $G$ induced by $A\cup B$. For a vertex $v\in J$, let $A_v$ ($B_v$) denote the set of edges in $A$ (resp. $B$) incident to $v$. Suppose that
\emph{  \begin{numitem1} \label{eq:ABnonempty} 
\emph{for each edge $fw\in A$, the set $B_w$ is nonempty, and similarly, for each $fw\in B$, the set $A_f$  is nonempty.}
  \end{numitem1} }
\noindent Then there exist \emph{nonzero} functions $\alpha:A\to \Rset_+$ and $\beta:B\to \Rset_+$ such that:

{\rm(i)} for each $v\in J$, ~$\sum(\alpha(e)\colon e\in A_v)=\sum(\beta(e)\colon e\in B_v)$;

{\rm(ii)} for each $f\in J\cap F$, the values $\alpha(e)$ for all edges $e$ in $A_f$ are equal, and similarly, for each $w\in J\cap W$, the values $\beta(e)$ for all $e\in B_w$ are equal.

\end{lemma}

\begin{proof}
Let $J_F:=J\cap F$ and $J_W:=J\cap W$. For $a\in\Rset_+^{J_F}$, denote by $\alpha^a$ the function on $A$ taking the value $a(f)$ for each edge $fw\in A_f$.  For $b\in\Rset_+^{J_W}$, denote by $\beta^b$ the function on $B$ taking the value $b(w)$ for each edge $fw\in B_w$. 

Define operators $p:\Rset_+^{J_F}\to\Rset_+^{J_W}$ and $q:\Rset^{J_W}\to\Rset_+^{J_F}$ as follows:
  \begin{numitem1}\label{eq:pq}
  \begin{itemize}
\item[(a)] for $a\in\Rset_+^{J_F}$ and $w\in J_W$, put $p(a)_w:=\sum(a(f)\colon  fw\in A_w)\,/\,|B_w|$;
\item[(b)] for $b\in\Rset_+^{J_W}$ and $f\in J_F$, put $q(b)_f:=\sum(b(w)\colon  fw\in B_f)\,/\, |A_f|$;
  \end{itemize}
  \end{numitem1}
and let $r$ be their composition $q\circ p$ (which maps $\Rset_+^{J_F}$ into itself). These $p,q,r$ are well defined since condition~\refeq{ABnonempty} implies  that for each $f\in J_F$,  $A_f\ne\emptyset$, and similarly for each $w\in J_W$, $B_w\ne\emptyset$ (though some $A_w$ ($w\in J_W$) or $B_f$ ($f\in J_f$) may be empty). 

Fix a number $\xi>0$, and consider the set $Q$ formed by the vectors $a\in\Rset_+^{J_F}$ such that $\sum(a(f)|A_f|\colon f\in J_F)=\xi$. This set forms a nonempty polytope (taking into account that each $a(f)$ is nonnegative and bounded, in view of $A_f\ne\emptyset$). 

Now using~\refeq{pq}(a), for $a\in\Rset_+^{J_F}$ and $b:=p(a)$, we obtain
 \begin{multline*} 
 \sum(b(w)|B_w|\colon w\in J_W)=\sum\nolimits_{w\in J_W} \sum(a(f)\colon fw\in A_w)               = \sum(a(f)|A_f|\colon f\in J_F).\quad
  \end{multline*}

Similarly, \refeq{pq}(b) implies that for $b\in\Rset_+^{J_W}$ and $a':=q(b)$,
  $$
  \sum(a'(f)|A_f|\colon f\in J_F)=\sum\nolimits_{f\in J_F}\sum(b(w)\colon
      fw\in B_f)=\sum(b(w)|B_w|\colon w\in J_W).
    $$
    
It follows that  for $a\in Q$, the vector $a':=r(a)=q(p(a))$ belongs to $Q$ as well. So the operator $r$ maps $Q$ into itself. Since $Q$ is a polytope and $r$ is, obviously, continuous, we can apply Brouwer fixed point theorem and conclude that there exists $a\in Q$ such that $r(a)=a$. Then the functions $\alpha=\alpha^a$ and $\beta=\beta^{p(a)}$ are as required in~(i)--(ii).
 \end{proof}

\noindent\textbf{Definition 4.} Refining the term ``active'' mentioned above, we call an edge $e=fw$ \emph{active for} $F$ if  $f\in F^+$ and $e\in D_f$, and \emph{active for} $W$ if $w\in W^+$ and $e\in H_w$. The subgraph of $G$ induced by the set of active edges is called the \emph{active graph} and denoted as $\Gamma=\Gamma(x)=(V_\Gamma,E_\Gamma)$. The directed counterpart of the active graph, denoted as $\orar\Gamma=\orar\Gamma(x)=(V_\Gamma,\orar {E}_\Gamma)$, is obtained by replacing each edge $e=fw\in E_\Gamma$ by the directed edge $\orar e=(f,w)$ if $e\in D_f$, and by the directed edge $\orar e=(w,f)$ if $e\in H_w$. A (weak) \emph{component} of $\Gamma$ (or $\orar\Gamma$) is a maximal subgraph $K$ in it such that any two vertices of $K$ are connected by a (not necessarily directed) path. 
  \medskip
  
Using Lemma~\ref{lm:plus-minus}, one can obtain the following
 \begin{lemma} \label{lm:reg_comp}
Let $K=(V_K,E_K)$ be a component of the active graph $\Gamma$. Then~\refeq{lp1}--\refeq{lp4} has a solution $(\phi,\psi)$ taking zero values on all vertices outside $K$.
 \end{lemma}
  \begin{proof}
~Take as $A$ the set $\cup(D_f\colon f\in F_K)$ and take as $B$ the set $\cup(H_w\colon w\in W_K)$, where $F_K:=F\cap V_K$ and $W_K:=W\cap V_K$. Property~\refeq{f-w} implies that $A,B$ obey condition~\refeq{ABnonempty}. Then, by Lemma~\ref{lm:plus-minus}, there exist $\alpha: A\to \Rset_+$ and $\beta:B\to\Rset_+$ satisfying properties~(i),(ii) in that lemma; moreover, as is seen from its proof, for any number $\xi>0$, the functions $\alpha,\beta$ can be chosen  so as to satisfy the equality $\sum(\alpha(e)\colon e\in A)=\xi$. Take $\xi:=\eps/2$, define $\phi_f:=\alpha(e)$ for $f\in F_K$ and $e\in D_f$, and define $\psi_w:=\beta(e)$ for $w\in W_K$ and $e\in H_w$. These $\phi,\psi$ extended by zero outside $K$ satisfy~\refeq{lp1}--\refeq{lp4}, as required.
 \end{proof}

Lemma~\ref{lm:reg_comp} has a sharper version. To state it, consider the directed active graph $\orar\Gamma=(V_\Gamma,\orar{E}_\Gamma)$ (see Definition~4), and let $\Lscr=\Lscr(x)$ denote the set of strong components in $\orar\Gamma$. (Recall that a \emph{strong component} of a directed graph is a maximal (by inclusion) subgraph $L=(V_L,E_L)$ in it such that $|V_L|>1$ and for any ordered pair $(u,v)$ in $V_L$, there is a directed path from $u$ to $v$. In other words, $L$ is a maximal connected subgraph representable as the union of directed cycles.) We write $F_L$ for $V_L\cap F$ and $W_L$ for $V_L\cap W$. There is a natural partial order on $\Lscr$ in which $L\in\Lscr$ is regarded as \emph{preceding} $L'\in\Lscr$ if $\orar\Gamma$ has a directed path going from $L$ to $L'$. We denote by $\Lrot=\Lrot(x)$ the set of \emph{maximal} components in $\Lscr$ under this order (i.e., those having no succeeding components). Note that~\refeq{f-w} implies that if $v$ is a vertex of $L\in\Lrot$, then any edge of $\orar\Gamma$ leaving $v$ belongs to $L$ as well.

We associate with $(\phi,\psi)\in \Rset_+^{F^+}\times\Rset_+^{W^+}$ the function $\orar\Delta=\orar\Delta^{\phi,\psi}$ on the set of active directed edges $\orar{E}_\Gamma$, taking the value $\phi_f$ on the edges $(f,w)\in D_f$, and $\psi_w$ on the edges $(w,f)\in H_w$ (Compare with $\Delta^{\phi,\psi}$ defined in~\refeq{Deltapp}). When $(\phi,\psi) \in \Pscr^\eps$, this $\orar\Delta$ represents a \emph{circulation} in $\orar\Gamma$, because the equalities in~\refeq{lp1} are equivalent to the conservation (or balance) conditions $\orar\Delta(\deltaout(f))=\orar\Delta(\deltain(f))$ for $f\in F^+$, and similarly for~\refeq{lp2} and $W^+$. Here $\deltain(v)$ ($\deltaout(v)$) denotes the set of edges in $\orar\Gamma$ entering (resp. leaving) a vertex $v$. The following fact is of use:
  \begin{numitem1} \label{eq:nonzero_path}
for $(\phi,\psi)\in\Pscr^\eps$, if vertices $u,v$ are connected in $\orar\Gamma$ by a directed path from $u$ to $v$, and if $(\phi,\psi)_u>0$, then $(\phi,\psi)_v>0$ as well.
 \end{numitem1}
(It suffices to see this for a single edge $e=(u,v)$ in $\orar\Gamma$. This edge belongs to the active head at $u$ (the directed version of either $H_u$ or $D_u$), and if $(\phi,\psi)_u>0$, then $\orar\Delta^{\phi,\psi}(e)>0$, implying $(\phi,\psi)_v>0$ by the balance condition at $v$.)

Now Lemma~\ref{lm:reg_comp} is strengthened as follows.  
  \begin{lemma} \label{lm:strong_comp}
For any maximal strong component $L\in\Lrot$, there exists a solution $(\phi,\psi)\in \Rset_+^{F^+}\times\Rset_+^{W^+}$ to system~\refeq{lp1}--\refeq{lp4} whose values outside $L$ are zeros, and this solution is unique.
  \end{lemma}
  \begin{proof}
~Consider  $L=(V_L,E_L)\in\Lrot$. As is mentioned above, for each vertex
$v\in V_L$, the set $\deltaout(v)$ of edges of $\orar\Gamma$ leaving $v$ is contained in $L$. So we can apply Lemma~\ref{lm:plus-minus} to the sets $A:=\cup(D_f\colon f\in F_L)$ and $B:=\cup(H_w\colon w\in W_L)$. This gives rise to a solution  $(\phi,\psi)\in\Pscr^\eps$ to~\refeq{lp1}--\refeq{lp4} having the support $\{v\colon (\phi,\psi)_v\ne 0\}$ within $L$. 
 
To see the uniqueness of $(\phi,\psi)$, we first observe that its support is exactly $V_L$. (This follows for the strong connectedness of $L$ and property~\refeq{nonzero_path}.)

Now suppose, for a contradiction, that there is another solution $(\phi',\psi')\in \Pscr^\eps$ having the support within $L$. Then one can choose $\lambda,\mu> 0$ such that the function $(\phi'',\psi''):=(\lambda\phi,\lambda\psi)-(\mu\phi',\mu\psi')$ is nonnegative, satisfies~\refeq{lp1}--\refeq{lp4}, and takes zero value on some vertex $f\in F_L$ or $w\in W_L$. But this is impossible by the observation above. 
  \end{proof}

On the other hand, one can show the following
  \begin{lemma} \label{lm:rest}
Any function $(\phi,\psi)\in \Rset_+^{F^+}\times\Rset_+^{W^+}$ satisfying~\refeq{lp1},\refeq{lp2},\refeq{lp4} takes zero value on all vertices not contained in maximal strong components of $\orar\Gamma$.
  \end{lemma}
  \begin{proof}
~Let $Y$ be the set of vertices occurring in components in $\Lrot$. Consider a function $(\phi,\psi)$ satisfying~\refeq{lp1},\refeq{lp2},\refeq{lp4}, and suppose, for a contradiction, that the set $Z$ of vertices $v\in V_\Gamma-Y$ with nonzero values $(\phi,\psi)_v$ is nonempty. Since $\Lrot$ consists of maximal components, the set  $\deltaout(Y)$ of edges of $\orar\Gamma$ going from $Y$ to $V_\Gamma-Y$ is empty.

On the other hand, for any vertex $u$ in $V_\Gamma-Y$, in particular, for $u\in Z$, there is a directed path $P$ going from $u$ to a maximal strong component of $\orar\Gamma$. Then, by~\refeq{nonzero_path}, $(\phi,\psi)_u>0$ implies $(\phi,\psi)_{v}>0$ for all vertices $v$ on this path; therefore, the edge $e$ of $P$ leaving $V_\Gamma-Y$ and entering $Y$ satisfies $\orar\Delta^{\phi,\psi}(e)>0$.

As a consequence, we obtain
    $$
  \orar\Delta(\deltain(Y))-\orar\Delta(\deltaout(Y))=\orar\Delta(\deltain(Y))>0,
  $$
where $\deltain(Y)$ is the set of edges of $\orar\Gamma$ entering $Y$. But the difference in the left hand side must be zero (as this is the sum, over the vertices $v$ of $Y$, of ``divergences''  $\orar\Delta(\deltain(v))-\orar\Delta(\deltaout(v))$, which are zeros). A contradiction. 
  \end{proof}

By Lemma~\ref{lm:strong_comp}, for each component $L\in\Lrot$, there is a unique solution to~\refeq{lp1}--\refeq{lp4} having the support in $L$ (which is, moreover, exactly $V_L$); we denote this solution as $(\phi_L^\eps,\psi_L^\eps)=(\phi_L^\eps(x),\psi_L^\eps(x))$. Then Lemmas~\ref{lm:strong_comp} and~\ref{lm:rest} imply
   \begin{corollary} \label{cor:vertP}
The vertices of polytope $\Pscr^\eps(x)$ (defined in~\refeq{polPscr}) are exactly the vectors $(\phi_L^\eps,\psi_L^\eps)$ for $L\in\Lrot(x)$. The supports of these vectors are pairwise disjoint.
  \end{corollary}

This enables us to specify the notion of rotations in our model.  
\medskip

\noindent\textbf{Definition 5.} By a $\Zset$-\emph{circulation} we mean an integer-valued function $\rho: E\to\Zset$ such that $\rho\ne 0$ and for each vertex $v\in V$, the sum $\rho(E_v)$ ($=\sum(\rho(e)\colon e\in E_v)$) is zero. Such a $\rho$ is called \emph{aligned} if $\rho(fw),\rho(fw')>0$ implies $\rho(fw)=\rho(fw')$, and, symmetrically, $\rho(fw),\rho(f'w)<0$ implies $\rho(fw)=\rho(f'w)$. We say that $\rho$ is \emph{normalized} if the g.c.d. of  values $\rho(e)$, $e\in E$, is one. We refer to a normalized aligned $\Zset$-circulation $\rho$ for which the subgraph of $G$ induced by the support $\{e\in E\colon \rho(e)\ne 0\}$ of $\rho$ is connected  as a \emph{pre-rotation} in $G$.
 \medskip
  
\noindent\textbf{Definition 6.} Consider a maximal strong component $L\in\Lrot$ and the corresponding vector $(\phi_L^\eps,\psi_L^\eps)\in\Pscr^\eps(x)$  with the support $V_L$ (existing by Lemma~\ref{lm:strong_comp}). Then there is a normalized aligned $Z$-circulation (viz. pre-rotation) proportional to the function $\Delta=\Delta^{\phi_L^\eps,\psi_L^\eps}$ on $E$ defined as in~\refeq{Deltapp}; we denote this circulation by $\rho_L$ and call it the \emph{rotation} for $x$ associated with $L$. Its analog for $\orar\Gamma$ is denoted as $\orar{\rho}_L$ (which for each $\orar e\in\orar E_\Gamma$, takes the value $|\Delta(e)|$). We say that a \emph{weight} $\lambda\in\Rset_+$ is \emph{admissible} for $L$, or for $\rho_L$, under the assignment $x$ if the following holds: 
  \begin{numitem1} \label{eq:adm_weight}
(a)~the function $x':=x+\lambda\rho_L$ is nonnegative and does not exceed the upper bound $b$, and~(b) for each $w\in W^+$, either $\pic_w(x')=\pic_w(x)$ and $H_w(x')\supseteq H_w(x)$, or $x'(e)=0$ for all edges $e$ in $\pic_w(x)$ (in the latter case, $x(e)=\lambda$, and the critical tie at $w$ becomes better: $\pic_w(x')>\pic_w(x)$). 
  \end{numitem1}
The maximal admissible weight $\lambda$ for $L$ is denoted by $\tau_L(x)$ and called the \emph{maximal rotational weight} for $L$, or for $\rho_L$, under $x$.
  \medskip
  
\noindent\textbf{Remark 1.} It is well known that all rotations arising in instances of stable allocation problem (SAP) are $0,\pm 1$ functions whose supports are formed by simple cycles. In contrast, in SMP, rotations may take large values (even exponentially large in $|E|$). A construction of ``big'' rotations will be given in Example~3 and Remark~3 in Sect.~\SEC{addit_prop}.  
 \medskip  
  
Lemma~\ref{lm:strong_comp} provides the existence of an admissible weight $\lambda>0$ for each $L\in\Lrot$. Clearly if $\lambda$ is admissible, then any $\lambda'$ in the interval $[0,\lambda]$ is admissible as well. Moreover, as a refinement of~\refeq{adm_weight} in case $\lambda<\tau_L$, one can see that
  \begin{numitem1} \label{eq:lambda<}
if $\lambda<\tau_L(x)$, then the transformation of $x$ into $x':=x+\lambda\rho_L$ preserves the active graph $\Gamma$ and: for each $w\in V^+$, we have $\pic_w(x')=\pic_w(x)$ and $H_w(x')=H_w(x)$, and for each $f\in F^+$, we have $H_f(x')=D_f(x)$.
  \end{numitem1} 
  
This implies that in case $\lambda<\tau_L(x)$, the rotation $\rho_L$ remains applicable to the new $x'$, but now the maximal weight for $\rho_L$ reduces to $\tau_L(x)-\lambda$ ($=\tau_L(x')$). 

Next, the importance of admissible weights is emphasized by the following
   \begin{lemma} \label{lm:tau_L}
For any $\lambda\le \tau_L(x)$, the assignment $x':=x+\lambda \Delta_L$ is stable and $x'\prec x$.
  \end{lemma}
  \begin{proof}
~We have $0\le x'\le b$ (by~\refeq{adm_weight}), and $|x'_v|=|x_v|$ for all $v\in V$  (by the definition of $\rho_L$). So $x'$ obeys both $b$ and $q$ and preserves the fully filled sets $F^=$ and $W^=$. 

To show the stability of $x'$, consider an edge $e=fw\in E$ with $x'(e)<b(e)$. We have to show that $e$ is not blocking for $x'$; equivalently, there is $v\in\{f,w\}$ such that the edge $e$ either belongs to the head $H_v(x')$ or occurs in $\Eend_v(x')$. Then $e$ is not interesting for $v$ under $x'$, and we are done. (Here it suffices to consider the case when both $f,w$ are fully filled; for otherwise Cleaning procedure (C) ensures that $x'_f=x_f$ and $x'_w=x_w$.)

To show this, we first suppose that $x'(e)\ne x(e)$. The case $x'(e)<x(e)$ is possible only if $w\in V_L$ and $e\in H_w(x)$. Then (cf.~\refeq{adm_weight}(b)) either $e\in H_w(x')$ or $e\in\pic_w(x)<\pic_w(x')$ (whence $e\in\Eend_w(x')$). And if $x'(e)>x(e)$, then $f\in V_L$ and $e\in D_f(x)$, implying $e\in H_f(x')$. Hence in both cases, $e$ is not blocking for $x'$.

Now let $x'(e)=x(e)<b(e)$. Then $e\notin E_L$. Since  $x$ is stable, we are in at least one of the following cases: (i) $e\in H_w(x)$; (ii) $e\in\Eend_w(x)$; (iii) $e\in H_f(x)$; (iv) $e\in\pi^i\subseteq \Eend_f(x)$ (where $w\in W^=$ in cases~(i),(ii), and $f\in F^=$ in cases (iii),(iv)). 

Note that if none of $f,w$ is in $V_L$, then $x'_f=x_f$ and $x'_w=x_w$, and the result follows. So assume that some of $f,w$ is in $V_L$ (while $e\notin E_L$). Moreover, we may assume that $w\in V_L$ in cases (i),(ii); for otherwise $x'_w=x_w$, and $e$ is not interesting for $w$ under $x'$ (since this is so under $x$). Similarly, we may assume that $f\in V_L$ in cases~(iii),(iv). Now consider these cases in detail. 

In case~(i) with $w\in V_L$, we have $e\in E_L$; so this is not the case.

In case~(ii) with $w\in V_L$, we have $\pic_w(x')\ge\pic_w(x)$. Then $e\in\Eend_w(x)$ implies $e\in\Eend_w(x')$, whence $e$ is not interesting for $w$ under $x'$.

In case~(iii) with $f\in V_L$, we have $e\in H_f(x)$ but $e\notin\tilde H_f(x)$ (for otherwise there would be $e\in\tilde H_f(x)=D_f(x)$, implying $x'(e)>x(e)$). This implies that $e\notin T_w(x)$, and therefore, $e\in H_w(x)$. (Here we appeal to Definition~3 and notice that $w$ cannot be deficit; for otherwise, we would have $D_f(x)=\emptyset$, whence $f$ should be inserted in $V^0$ by Cleaning procedure.) Then we are as in case~(i).

Finally, in case~(iv) with $f\in V_L$, compare $\pi^i\in\Pi_f$ (containing $e$) with $\pic_f(x')=\pi^j$ (containing $H_f(x')=D_f(x)$). If $\pi^i<\pi^j$, we are done. So assume that $\pi^i\ge \pi^j$. The vertex $w$ is not deficit. Also $e\notin D_f(x)$ (otherwise we obtain $x'(e)>x(e)$). It follows (cf. Definition~3) that the critical tie $\pic_w(x)$ must be better than the tie in $\Pi_w$ containing $e$, i.e. $e\in\Eend_w(x)$. This implies that $e\in\Eend_w(x')$, and $e$ is not interesting for $w$ under $x'$.

Thus, in all cases, $e$ is not blocking for $x'$, as required. The fact $x'\prec x$ is easy.
  \end{proof}

In fact, different rotations in $\Rscr(x):=\{\rho_L\colon L\in\Lrot(x)\}$ can be applied independently, and the following generalization of Lemma~\ref{lm:tau_L} is valid.
  \begin{prop} \label{pr:many}
Let $\lambda:\Lrot(x)\to\Rset_+$ satisfy $\lambda(L)\le\tau_L(x)$ for all $L\in\Lrot(x)$. Then $x':= x+\sum(\lambda(L)\rho_L\colon L\in\Lrot(x))$ is a stable assignment.
 \end{prop}
  \begin{proof}
This is shown by use of Lemma~\ref{lm:tau_L} and the fact that the maximal strong components of the active graph are pairwise disjoint (the latter implies, in particular, that $x'$ does not exceed upper bounds and quotas). More precisely, we have to show that each edge $e=fw\in E$ with $x'(e)<b(e)$ is non-blocking for $x'$. This is immediate if both vertices $f,w$ belong to the same component in $\Lrot$ (by appealing to Lemma~\ref{lm:tau_L}), or if  both equalities $x'_f=x_f$ and $x'_w=x_w$ take place (since $e$ is not blocking for $x$, implying that $e$ is not blocking for $x'$ as well).

So we may assume that $e$ belongs to no component in $\Lrot$ (implying $x'(e)=x(e)$) and that either $x'_f\ne x_f$ or $x'_w\ne x_w$ (or both) takes place.

Arguing as in the proof of Lemma~\ref{lm:tau_L}, we get into at least one of cases~(i)--(iv) examined there; moreover, we have $w\in V_L$ in case~(i) or~(ii), and $f\in V_L$ in case~(iii) or~(iv), where $L\in\Lrot$. In cases~(i),(ii) (when $w\in V_L$), our analysis is similar to that in the previous proof. And in cases~(iii),(iv) (when $f\in V_L$), if the other end $w$ of $e$ does not belong to any other component in $\Lrot$, then we again argue as in the previous proof. And if $w$ belongs to another component $L'\in\Lrot$, then we can switch to consideration of this  $L'$ rather than $L$, where we get into case~(i) or~(ii).
  \end{proof}

As a consequence of this proposition, we obtain the following
 \begin{corollary} \label{cor:Pgood}
For a sufficiently small $\eps>0$, any assignment $x':=x+\Delta^{\phi,\psi}$ with $(\phi,\psi)\in\Pscr^\eps(x)$ (defined in~\refeq{polPscr}) is stable, i.e. $x'\in X^\eps(x)$ (defined in~\refeq{Xeps}) .
  \end{corollary}

This gives one direction in assertion~(a) of Theorem~\ref{tm:close}. The other (easier) direction, saying that any stable assignment $x'\in\Sscr$ close to $x$ and satisfying $x'\prec x$ belongs to $\{ x+\Delta^{\phi,\psi}\colon (\phi,\psi)\in
\Pscr^\eps(x)\}$ for an appropriate $\eps$, can be concluded from our discussion in the beginning of this section (including Lemmas~\ref{lm:xwyw} and~\ref{lm:xfyf} and Cleaning procedure).

Now we show part (b) of this theorem (which will lead to the nice property that any stable $y\prec x$ can be reached from $x$ by a sequence of rotations).

 \begin{lemma} \label{lm:second}
If $x,y\in\Sscr$ are such that $y\prec x$, then for a sufficiently small $\eps>0$, there exists $x'\in X^\eps(x)$ between $x$ and $y$, namely, $y\preceq x'\prec x$. 
 \end{lemma}
 \begin{proof}
Define $E^>:=\{e\in E\colon y(e)>x(e)\}$ and $E^<:=\{e\in E\colon y(e)<x(e)\}$.
For $v\in V$, define $E^>_v:=E^>\cap E_v$ and $E^<_v:=E^<\cap E_v$. Clearly each edge $fw$ in $E^>\cup E^<$ connects fully filled vertices: $f\in F^=$ and $w\in W^=$. Let $V^\gtrless$ denote the set of vertices covered by the edges in $E^>$ (and $E^<$).

First of all we specify the structure and interrelation of $E^>_v$ and $E^<_v$ for $v\in V^\gtrless$. We observe the following:
  \begin{numitem1} \label{eq:f><}
for $f\in V^\gtrless\cap F$, the equality $C_f(x_f\vee y_f)=x_f$ (in view of $x_f\succ y_f$) implies $E^>_f\subseteq H_f(x)\cup \Eend_f(x)$ and $E^<_f\subseteq \Ebeg_f(x)\cup\{\pic_f(x)\}$; and
  \end{numitem1}
  \begin{numitem1} \label{eq:w><}
for $w\in V^\gtrless\cap W$, the equality $C_w(x_w\vee y_w)=y_w$ (in view of $y\succ_W x$) implies that either $E^>_w\subseteq T_w(x)$ and $H_w(x)\subseteq E^<_w\subseteq \pic_w(x)$ (when $\pic_w(y)=\pic_w(x)$), or $E^>_w\subseteq \Ebeg_w(x)$ and $E^<_w=\{e\in \pic_w(x)\colon x(e)>0\}$ (when $\pic_w(y)>\pic_w(x)$). 
  \end{numitem1}
From there relations one can conclude that  
  \begin{numitem1} \label{eq:E>f}
for any $f\in V^\gtrless\cap F$, there holds $E^>_f\subseteq  \cup(D^i_f(x)\colon  \pi^i\in \hat \Pi_f(x) )$. 
  \end{numitem1}
(Where $D^i_f(x)$ and $\hat \Pi_f(x)$ are defined in Definition~3.)
Using these observations, we construct sets $A_v$ of active edges for $x$ and $v\in V^\gtrless$ by the following procedure. Initially, we fix an arbitrary vertex $f=f_0$ in $V^\gtrless\cap F$. The inclusion in~\refeq{E>f} implies  $D_f(x)\ne\emptyset$. Define $A_f:=D_f(x)$ and $W(A_f):=\{w\colon fw\in A_f\}$. We assert that
  \begin{equation} \label{eq:WAf}
  W(A_f)\subseteq V^\gtrless \cap W.
  \end{equation}
Indeed, suppose this is not so, and let $e=fw\in A_f$ be such that $w\notin V^\gtrless$. Then $y_w$ ($=y\rest{E_w}$) coincides with $x_w$. By the definition of $D_f(x)$, $e$ is unsaturated and belongs to the tail $T_w(x)$. Therefore, $e$ is interesting for $w$ under $x$ and $y$. At the same time, since $D_f(x)$ is the most preferable set among all $D^i_f(x)$ for $\pi^i\in\hat \Pi_f(x)$, the inclusion in~\refeq{E>f} implies that the tie in $\Pi_f$ containing $A_f$ is not worse than the critical tie $\pic_f(y)$ (which contains a nonempty subset of $E^>_f$). This implies that $e$ is interesting for $f$ under $y$ (taking into account that $y(e)=x(e)$, by the supposition). But then the edge $e$ is blocking for $y$; a contradiction.

Now using~\refeq{WAf}, we proceed as follows. For each $w\in W(A_f)$, assign $A_w:=H_w(x)$. Then $A_w\subseteq E^<_w$ (cf.~\refeq{w><}). For each $e=fw\in A_w$, the vertex $f$ belongs to $V^\gtrless$ (since $e\in E^<$), and we handle $f$ in a way similar to that applied to $f_0$ above. And so on.

As a result, this procedure constructs a collection $\Qscr$ of sets $A_v$ for certain vertices $v$ in $V^\gtrless$. Here if $v=f\in F$, then $A_f$ coincides with $D_f(x)$, and if $v=w\in W$, then $A_w=H_w$. Upon termination of the procedure, for each edge $fw\in A_f\in\Qscr$, the collection $\Qscr$ contains the set $A_w$, and similarly, for each edge $fw\in A_w\in \Qscr$, $\Qscr$ contains $A_f$. In other words, $\Qscr$ determines a subgraph $\orar\Gamma'$ of the active graph $\orar\Gamma(x)$, which is closed under the reachability by directed paths. Then $\orar\Gamma'$ contains a maximal strong component $L$ of $\orar\Gamma(x)$. We fix $L$, take the rotation $\rho_L$ along with a sufficiently small weight $\lambda>0$, and define $x':=x+\lambda\rho_L$. We assert that this $x'$ is as required.

To see this, we first observe that by the construction and explanations above,
  \begin{numitem1} \label{eq:AwD}
each $w\in V_L\cap W$ satisfies $A_w=H_w\subseteq E^<_w$, and each $e\in A_w$ satisfies $y(e)<x'(e)<x(e)$.
  \end{numitem1}
  
In turn, for each $f\in V_L\cap F$, the edges $e\in A_f$ satisfy $x(e)<x'(e)<b(e)$. We need somewhat stronger relations involving $y$. More precisely, we assert that
  \begin{numitem1} \label{eq:AfD}
for each $f\in V_L\cap F$, any edge $e=fw\in A_f$ ($=D_f(x)$) with $y(e)<b(e)$ belongs to either $E^>\cup H_w(y)$ or $\Eend_w(y)$.
  \end{numitem1}
Indeed, since $e$ is not blocking for $y$, $e$ belongs to at most one tail among $T_f(y)$ and $T_w(y)$ (cf.~\refeq{blocking}); therefore, at least one of the following takes place: (i) $e\in H_f(y)$; (ii) $e\in H_w(y)$; (iii) $e$ belongs to $\Eend_f(y)$; and (iv) $e$ belongs to $\Eend_w(y)$. 

In cases~(ii) and~(iv), we are done. Assume that we are not in these cases. Then $e\in T_w(y)$. Also, by the definition of $D_f(x)$, we have $e\in T_w(x)$. Let $\pi^i$ and $\pi^j$ be the ties in $\Pi_f$ containing $D_f(x)$ and $H_f(y)$, respectively. In case~(iii), $\pi^j$ is better than $\pi^i$ (containing $e$), which is impossible. Finally, in case~(i), we have $e\in D_f(x)\cap H_f(y)$, whence $i=j$. Suppose that $\pi^i=\pic_f(x)$. Then $y_f\preceq x_f$ and $y_f\ne x_f$ (in view of $f\in V^\gtrless$) imply that $H_f(y)\subseteq H_f(x)$ and $H_f(y)\subseteq E^>$, whence $e\in E^>$. And if $\pi^i$ is worse than $\pic_f(x)$, then $0=x(e)<y(e)$. In both cases, we obtain $e\in E^>$, completing the examination of~\refeq{AfD}.

Now considering the relations in~\refeq{AwD} and~\refeq{AfD}, we can conclude that $C_w(y_w\vee x'_w)=y_w$. Thus, $y_w\succeq_w x'_w$, and this relation holds for all $w\in V^\gtrless\cap W$, yielding $y\succeq_W x'$. By the polarity, this implies $y\preceq_F x'$, proving the lemma. 
 \end{proof} 

This completes the proof of Theorem~\ref{tm:close}. \hfill$\qed\; \qed$
 \medskip

In fact, the above analysis enables us to establish additional useful properties.
First of all the following sharper version of Lemma~\ref{lm:second} can be extracted from the above proof (it will be used in Sect~\SEC{addit_prop}). 
 \begin{lemma} \label{lm:strength_xp}
Let $x,y\in\Sscr$ and $y\prec x$. Then there exists a maximal strong component $L$ in $\orar\Gamma(x)$ and a number $\lambda\le \tau_L(x)$ such that the assignment $x':=x+\lambda\rho_L$ satisfies $y\preceq x'\prec x$. The maximum possible weight $\lambda$ (subject to $y\preceq x'$) is equal to the minimum between the number $\tau_L(x)$, the values $(y(e)-x(e))/\rho_L(e)$ over $e\in E_L$ such that $\rho_L(e)>0$ and $y(e)>x(e)$, and the values $(x(e)-y(e))/|\rho_L(e)|$ over $e\in E_L$ such that $\rho_L(e)<0$ and $x(e)>y(e)$. \hfill$\qed$
  \end{lemma}

Also, using Lemma~\ref{lm:second}, one can conclude with the following
   \begin{corollary} \label{cor:xmax}
For $x\in\Sscr$, the following are equivalent:
  \begin{itemize}
  \item[\rm(i)]  $x\ne\xmax$;
  \item[\rm(ii)] system~\refeq{lp1}--\refeq{lp4} has a solution with $\eps>0$;
  \item[\rm(iii)] the set of active edges (see Definition~4) is nonempty;
  \item[\rm(iv)] the active graph $\orar\Gamma(x)$ has a strongly connected component.
   \end{itemize}
  \end{corollary}
  
Indeed, if (i) is valid, then for $y:=\xmax$, we have $y\prec x$, and Lemma~\ref{lm:second} implies validity of~(ii). Conversely, if system~\refeq{lp1}--\refeq{lp4} is solvable for $\eps>0$, then (by Lemma~\ref{lm:tau_L}) there exists a stable $x'$ such that $x'\prec x$, yielding~(i). The equivalence of~(iii) and~(iv) can be concluded from Cleaning procedure and~\refeq{f-w}. The implication (iii)$\to$(ii) follows from Lemma~\ref{lm:reg_comp}. The converse implication (ii)$\to$(iii) can be concluded from the discussion in the beginning part of this section.

\section{Additional properties of rotations} \label{sec:addit_prop}

In this section we discuss additional (in particular, computational) aspects related to rotations and expose an appealing example of rotations.

We will refer to the update $x\rightsquigarrow x':=x+\lambda\rho_L$, where $L\in\Lrot(x)$ and $\lambda$ is an admissible weight for $x$ and the rotation $\rho_L$, as the rotational \emph{shift} of weight $\lambda$ along $\rho_L$ applied to $x$. We also say that $x'$ is obtained from $x$ by the shift by $\lambda\rho_L$. If $\lambda=\tau_L(x)$, the shift is called \emph{full} (w.r.t. $x$). As before, we write $F_L$ for $F\cap V_L$, and $W_L$ for $W\cap V_L$. One can see that
  \begin{numitem1} \label{eq:tau_L}
the maximal admissible weight $\tau_L(x)$ is equal to the minimum of the following values: (a)  $x(e)/|\rho_L(e)|$ over all $e\in H_w$, $w\in W_L$; (b)  $(b(e)-x(e))/\rho_L(e)$ over all $e\in D_f$, $f\in F_L$; and (c) $(x(e)-x(e'))/(|\rho_L(e)|+\rho_L(e'))$ over all pairs $e,e'\in E_L$ such that $e\in H_w$ and $e'\in \pic_w(x)-H_w$ among $w\in W_L$.
  \end{numitem1}
(Here (a) and (b) ensure that $x'$ is nonnegative and does not exceed the upper bounds $b$, and (c) provides (when the critical tie preserves) the inclusion $H_w(x)\subseteq H_w(x')$ for $w\in W_L$. In other words, in case~(c), the new (increased) assignment in any edge of $\pic_w-H_w$ should not exceed the new (decreased) assignment on an edge of $H_w$.)
 \smallskip
 
Given a stable $x$, the extraction of critical ties and their heads, and, further, the construction of the active graph $\orar\Gamma$, take linear time $O(|E|)$. The task of finding the strong components in $\orar\Gamma$ can also be solved in linear time (e.g. using the first linear time algorithm for this task in~\cite{karz}). Then we can easily extract the set $\Lrot$ of maximal strong components. Next, for each $L\in\Lrot$, we efficiently construct the rotation $\rho_L$. To do this, we first solve system~\refeq{lp1}--\refeq{lp4} (in which $F^+$ and $W^+$ are replaced by $F\cap V_L$ and $W\cap V_L$, respectively, and one puts $\eps=1$, say). And second, we transform the obtained solution $(\phi,\psi)\in\Qset_+^{F^+}\times\Qset_+^{W^+}$ into the rotation $\rho_L\in\Zset^E$. 

(Here the system has size $(n+1)\times n$ (where $n:=|V|$) and integer entries $a_{ij}$ with $|a_{ij}|\le n$. It is solved by Gaussian eliminations in time $O(n^3)$ (and all intermediate data during the process are rational numbers with numerators and denominators whose encoding sizes (viz. numbers of bits to record) are of order $O(n^2\log n)$, as is shown by Edmonds for a general case, see e.g.~\cite[Sec.~4.3]{KV} or~\cite{schr}. Now to form the desired rotation, the obtained $(\phi,\psi)$ is scaled to get all entries integer-valued, and then we get rid of the common divisor by use of Euclidean algorithm (which takes $O(n^3\log n)$ time in total, cf.~\cite[Sec.~4.2]{KV}). At present we cannot offer a more combinatorial efficient method to construct the rotation related to a strong component $L$.)  

Finally, using~\refeq{tau_L}, one can compute the maximal rotational weights $\tau_L$ for all rotations $\rho_L$ in time $O(|E|)$ (since the sets $V_L$ for $L\in\Lrot$ are pairwise disjoint). 

Summarizing the above observations, we obtain the following complexity result.
  \begin{prop} \label{pr:time}
Given $x\in\Sscr$, we can fulfill, in time polynomial in $|E|$, the following tasks: {\rm(a)} recognize whether $x=\xmax$, and {\rm(b)} in case $x\ne\xmax$, find all rotations $\rho_L$, $L\in\Lrot(x)$, and their maximal admissible weights $\tau_L(x)$. Each entry in a rotation has the encoding size bounded by a polynomial in $|E|$, and similarly for the encoding sizes of intermediate data appeared in the process of computing the rotations.
  \end{prop}

Next, based on the analysis from the previous section, one can estimate the number of rotational transformations (shifts) applied in the process of moving from the minimal stable assignment $\xmin$ or from an arbitrary stable $x\in \Sscr$ to another stable assignment or to $\xmax$. Here the following observations are of use (cf.~\refeq{adm_weight} and~\refeq{tau_L}):
  \begin{numitem1} \label{eq:rotat_shift}
Let $x'$ be obtained from $x\in\Sscr$ by the shift of weight $\lambda$ along a rotation $\rho_L$; then:
  \begin{itemize}
\item[(a)] for each $f\in F^=$, either $\pic_f(x')=\pic_f(x)$ and $H_f(x')\subseteq H_f(x)$, or $\pic_f(x')$ is worse than $\pic_f(x)$; in its turn, for each $w\in W^=$, either $\pic_f(x')=\pic_f(x)$ and $H_w(x')\supseteq H_w(x)$, or $\pic_w(x')$ is better than $\pic_w(x)$;
\item[(b)] for $\lambda=\tau_L(x)$, at least one of the following takes place: (i) for some $f\in F^=$, either $\pic_f(x')=\pic_f(x)$ and $H_f(x')\subset H_f(x)$, or $\pic_f(x')<\pic_f(x)$; (ii) for some $w\in W^=$, either $\pic_w(x')=\pic_w(x)$ and $H_w(x')\supset H_w(x)$, or $\pic_w(x')>\pic_w(x)$.
  \end{itemize}
  \end{numitem1}
  
Note that this monotonicity behavior leads to the following:
  \begin{numitem1} \label{eq:monoton}
in the process of rotational shifts, for each edge $e\in fw\in E$, the updates of  current assignments $x(e)$ can be chronologically divided into two phases: at the first phase, the value at $e$ is monotone non-decreasing (when $e$ occurs in ``heads'' concerning the part $F$); and at the second phase, the value at $e$ is monotone non-increasing (when $e$ occurs in ``heads'' concerning the part $W$).
  \end{numitem1}
  
Also, using~\refeq{rotat_shift}, we obtain the following
  \begin{lemma} \label{lm:2E}
{\rm(i)} If  a stable assignment is constructed from another one by a sequence of full rotational shifts, then the number of shifts is at most $2|E|$.
 
{\rm(ii)} Any $x\in\Sscr$ can be constructed from $\xmin$ by $2|E|$ rotational shifts.
 \end{lemma}

\begin{proof}
To show~(i), consider corresponding cases in~\refeq{rotat_shift} and associate with a full shift $x\rightsquigarrow x'$ in a given sequence one edge $e=fw$ as follows. Choose $e\in H_f(x)-H_f(x')$ in case $H_f(x')\subset H_f(x)$; choose $e\in H_f(x)$ in case $\pic_f(x')<\pic_f(x)$; choose $e\in H_w(x')-H_w(x)$ in case $H_w(x')\supset H_w(x)$; and choose $e\in H_w(x')$ in case $\pic_w(x')>\pic_w(x)$. Then each edge is associated with at most two full shifts in the sequence, whence~(i) follows.   

To show~(ii), it is useful to note that, given $x\in\Sscr$, the task of constructing $x$ by rotational shifts from $\xmin$ is equivalent to that of constructing $\xmax$ from $x$ (for we can ``swap'' the parts $F$ and $W$ and ``reverse'' rotations). So we can consider the latter problem. (By the way, if needed, we can efficiently verify whether a function $x$ on $E$ obeying $b,q$ is indeed a stable assignment, by checking the absence of blocking edges.) 

It can be seen (e.g., from Lemma~\ref{lm:strength_xp}) that $\xmax$ can be obtained from any $x\in \Sscr$ by a finite sequence of rotational shifts with some admissible weights. If in this sequence, one and the same rotation $\rho_L$ (equivalently, the same maximal strong component $L$ in current active graphs) occurs several times, then we can combine the shifts related to these occurrences into one shift along $\rho_L$, summing up the rotational weights. (Here we use the fact that if a shift along $\rho_L$ is not full, then $L$ continues to be a maximal strong component in the current active graph, cf. the proof of Proposition~\ref{pr:many}.) Thus, the initial sequence can be transformed into another one, going from $x$ to $\xmax$ as well, so that each rotation in it occur exactly once. Then it is easy to see that all rotational shifts are full, and now the result follows from~(i).
   \end{proof}
 
\noindent\textbf{Remark 2.} Note that in general the set $\Sscr$ of stable assignments in SMP need not be convex in $\Rset^E$. (This is so already for SAP with  $c\equiv 1$ and $q\equiv 1$. For example, take as $G$ the 6-cycle $(v_0,e_1,v_1,\ldots,e_6,v_6=v_0)$ and add to it edge $a=v_1v_4$. Assign $e_i<e_{i+1}$ for $i=1,\ldots,6$ (letting $e_7:=e_1$), $e_1<a<e_2$ and $e_4<a<e_5$. Then the incidence vectors of $\{e_1,e_3,e_5\}$ and $\{e_2,e_4,e_6\}$ give stable allocations but their half-sum $x$ does not, since the edge $a$ is blocking for $x$.)  However, using Proposition~\ref{pr:many}, one can give the following geometric description of this set via rotations. Let $\Sscr^{\rm full}$ be the set of stable assignments that can be obtained from $\xmin$ by a sequence of full rotational shifts. Then $\Sscr$ is the union, over $x\in\Sscr^{\rm full}$, of the rectangular parallelepipeds $x+Q(x)$, where $Q(x)$ is the Minkowsky sum of orthogonal segments $[0,\tau_L\rho_L]$, $L\in\Lrot(x)$.
  \medskip   

\noindent\textbf{Example 3.} As is mentioned in Remark~1 from Sect.~\SEC{rotat}, there exist rotations with ``big'' values. An example is illustrated in Fig.~\ref{fig:fract}.
Here we draw an active graph $\Gamma$ with vertex parts $F=\{f_1,f_2,f_3\}$ and $W=\{w_1,w_2,w_3\}$, the edges active for $F$ ($W$) are drawn by bold (resp. dotted) lines, and the numbers on edges point out values of rotation $\rho$ (this is a normalized aligned $Z$-circulation); they vary from $-8$ through 7. 

Note that this rotation appears in the following instance of SMP: $G=\Gamma$; the weak orders at vertices of $G$ are defined as:
 \begin{eqnarray*}
 \mbox{for vertex $f_i$:} && f_iw_1\sim f_iw_2\sim f_iw_3,\quad i=1,2,3; \\ 
 \mbox{for vertex $w_1$:} && f_1w_1> f_3w_1>f_2w_1; \\ 
 \mbox{for vertex $w_2$:} && f_1w_2> f_2w_2>f_3w_2; \\ 
 \mbox{for vertex $w_3$:} && f_2w_3> f_1w_3\sim f_3w_3. 
 \end{eqnarray*}
For two arbitrarily chosen numbers $0<a<b$,  the upper bound $b(e)$ for each edge $e$ is equal to $b$, and the quota $q(v)$ of each vertex $v$ is equal to $3a$. Let $x(e):=a$ for all edges $e$ of $G$. One can see that: all vertices in $F^=(x)=F$ and $W^=(x)=W$ are fully filled by $x$; for all $w_j\in W$, the critical ties $\pic_{w_j}(x)$ consist just of the dotted edges and they coincide with the heads $H_{w_j}(x)$; for all $f_i\in F$, the heads $D_{f_i}(x)$ consist of the bold edges. As a result, $x$ is stable. All edges are unsaturated by $x$, and one can see that the directed active graph $\orar\Gamma(x)$ has a unique strong component $L$ involving all edges. Then $\rho_L=\rho$. Applying~\refeq{tau_L}, we obtain that the maximum rotational weight $\tau_L(x)$ is equal to the minimum of $a/8$ and $(b-a)/7$.

 \begin{figure}[htb]
\begin{center}
\vspace{-0.1cm}
\includegraphics[scale=0.8]{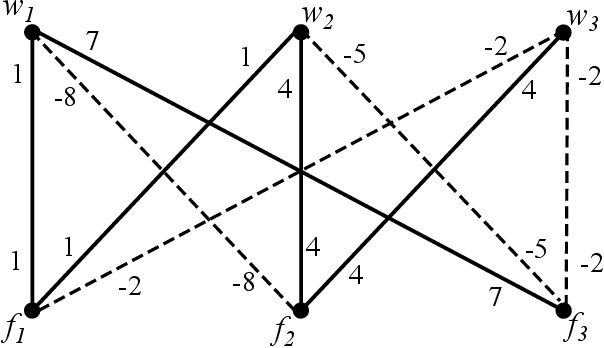}
\end{center}
\vspace{-0.5cm} 
\caption{A  rotation with ``big'' values}
 \label{fig:fract}
\end{figure}

\medskip
\noindent\textbf{Remark~3.} We can extend the above example as follows. Take copies $\Gamma^1,\ldots,\Gamma^k$ of the graph $\Gamma$, denoting the copies of vertices $f_i$ and $w_i$ ($i=1,2,3$) in the graph $\Gamma^j$ ($j=1,\ldots,k$) as $f_i^j$ and $w_i^j$, respectively. For $j=1,\ldots,k-1$, identify the vertex $w_1^j$ of $\Gamma^j$ with the vertex $w_3^{j+1}$ of $\Gamma^{j+1}$, denoting the obtained vertex as $\hat w^j$. Let $\hat\Gamma$ be the resulting graph, the union of $\Gamma^1,\ldots,\Gamma^k$. (In particular, $\hat w^1$ has three $F$-active (``bold'') edges, namely, $f_1^1w_1^1,\, f_3^1w_1^1,\, f^2_2w_3^2$, and three $W$-active (``dotted'') edges, namely, $f_2^1w_1^1,\, f_1^2w_3^2,\, f_3^2w_3^2$ (where $w_1^1=w_3^2=\hat w^1$).) Let $\hat \rho$ be the rotation generated by $\hat\Gamma$, and $\rho^j$ its restriction on $\Gamma^j$. One can see that $\rho^1$ coincides with the rotation $\rho$ for $\Gamma$ (taking into account that $\rho^1$ is determined by the balance and align conditions on the five vertices $f_1^1,f_2^1,f_3^1,w_2^1,w_3^1$). Then $\hat\rho(f_2^1\hat w^1)=\rho(f_2w_1)=-8$, which implies that $\hat\rho(f_1^2\hat w^1)=\hat\rho(f_3^2\hat w^1) =\hat\rho(f_2^1\hat w^1)=-8$. Since $\rho(f_1w_3)=\rho(f_3w_3)=-2$, we obtain that $\rho^2$ is expressed as $4\rho$ (under identifying $\Gamma^2$ with $\Gamma$). Arguing similarly for $j=2$, we conclude that $\rho^3$ is expressed as $4\rho^2$, or as $16\rho$. Eventually, $\rho^k$ is expressed as $4^{k-1}\rho$. 

Thus, the rotation $\hat\rho$ in our construction has the $\ell_\infty$-norm $2\cdot 4^{k-1}$, whereas the graph $\hat\Gamma$ has $5k+1$ vertices, thus providing an exponential grow of the rotation norms.
 
  
\section{The poset of rotations} \label{sec:poset_rot}

In this section we explain that the stable assignments in our mixed model admit an efficient representation, via the so-called closed sets (weighted analogs of ideals) in a poset generated by the rotations. 

We have seen that any stable assignment $x\in\Sscr$ can be obtained from  $\xmin$ (which is the best for $F$ and the worst for $W$) by use of a sequence of $O(|E|)$ rotational shifts, see Lemma~\ref{lm:2E}. The case $x=\xmax$ is of an especial interest for us.
 \medskip
 
\noindent\textbf{Definition 7.}  Let $T$ be a sequence $x_0,x_1,\ldots,x_N$ of stable assignments such that each $x_i$ is obtained from $x_{i-1}$ by a rotational shift, i.e., $x_i=x_{i-1}+\lambda_i\rho_{L_i}$ (where $\rho_{L_i}$ is the rotation generated by an element (maximal strong component) $L_i$ of the set $\Lrot(x_{i-1})$, and $\lambda_i$ is an admissible weight: $\lambda_i\le\tau_{L_i}(x_{i-1})$; see Definition~6. We say that $T$ is \emph{non-expensive} is all rotations $\rho_{L_i}$ are different (as vectors in $\Rset^E$). In particular, $T$ is non-expensive if all involved shifts are full, in the sense that each $\lambda_i$ is the maximal rotational weight for $L_i$ at the moment it applies. A non-expensive $T$ with $x_0=\xmin$ and $x_N=\xmax$ is of most interest for us; such a $T$ is called a (complete) \emph{route}. The set of all routs is denoted by $\Tscr^0$. 

In what follows we will deal with only non-expensive sequences $T=(x_0,\ldots,x_N)$. When $x_0=\xmin$ and $x=x_N$ is arbitrary (resp. $x=x_0$ is arbitrary and $x_N=\xmax$), we call $T$ a \emph{lower} (resp. \emph{upper}) \emph{semi-route}. The set of lower (resp. upper) semi-routes for $x\in \Sscr$ is denoted by $\Tscrlow(x)$ (resp. $\Tscrup(x)$). The set of rotations used to form a route or semi-route $T$ is denoted as $\Rscr(T)$. The set of pairs $(\lambda,\rho)$ formed by rotations $\rho\in\Rscr(T)$ and their weights $\lambda$ used for $T$ is denoted by $\Omega(T)$.
 \medskip

An important fact is that all routes use the same set of rotations along with their weights. In other words, a strong component $L$ determining the rotation $\rho_L$ which is used in one route should necessarily appear in any other one and, moreover, each rotation uses the same weight in all occurrences. Initially the phenomenon of this sort was revealed for rotations in the stable marriage problem in~\cite{IL}; subsequently, this was established for other stability problems, in particular, for SAP in~\cite{DM}.

Borrowing a method of proof for SAP from~\cite{DM}, we show a similar property for SMP.
 \begin{prop} \label{pr:equal_rot}
The set of pairs $\Omega(T)$ is the same for all routes $T\in\Tscr^0$.
  \end{prop}
  \begin{proof}
~Let us say that a stable assignment $x\in\Sscr$ is \emph{good} if the set $\Omega(T)$ is the same for all upper semi-routes $T\in\Tscrup(x)$; otherwise $x$ is called \emph{bad}. We have to show that $\xmin$ is good. 

Suppose this is not so, and consider a maximal bad assignment $x$, in the sense that any $x'$ obtained from $x$ by a rotational shift of maximal admissible weight is already good. Since $x$ is bad, in the set $\Tscrup(x)$, there are two semi-routes $T'=(x,x',\ldots,\xmax)$ and $T''=(x,x'',\ldots,\xmax)$ such that $\Omega(T')\ne\Omega(T'')$.

Let $x'$ and $x''$ (occurring in $T',T''$) be obtained from $x$ by rotational shifts by $\lambda'\rho_{L'}$ and $\lambda''\rho_{L''}$, respectively. These shifts involve different maximal strong components $L'$ and $L''$ which are simultaneously contained in the active graph $\orar\Gamma(x)$. These components are disjoint, and  the above rotational shifts commute. Then we can form in $\Tscrup(x)$ two semi-routes; namely, $\tilde T'$ starting with the triple $x,x',x'+\lambda''\rho_{L''}$, and $\tilde T''$ starting with the triple $x,x'',x''+\lambda'\rho_{L'}$, which further continue using the same upper semi-route $\tilde T$ starting with $x+\lambda'\rho_{L'}+\lambda''\rho_{L''}$. 

Now since the assignment $x'$ is good (in view of the maximal choice of $x$), the parts of the semi-routes $T'$ and $\tilde T'$ going from $x'$ to $\xmax$ must determine the same sets of pairs $\Omega$, and similarly for the parts of $T''$ and $\tilde T''$ going from $x''$. But then $\Omega(T')=\{(\lambda',\rho_{L'}),\, (\lambda'',\rho_{L''})\}\cup \Omega(\tilde T)=\Omega(T'')$, contrary to the choice of $T',T''$.  
  \end{proof}
  
The above proposition we can extended as follows (by arguing in a similar way).
 \begin{prop} \label{pr:gomogen}
For any $x,y\in\Sscr$ with $x\succ y$, the sets of pairs $(\lambda,\rho)$ used to form non-expensive sequences (``middle semi-routes'') from $x$ to $y$ are the same. \hfill$\qed$
 \end{prop}

Next, relying on Proposition~\ref{pr:equal_rot}, we can denote the set of rotations generating shifts in an arbitrary (complete) route simply as $\Rscr$. The maximal admissible weight of a rotation $\rho=\rho_L\in\Rscr$ (which does not depend on a route) is denoted by $\tau_L$, or by $\tau(\rho)$. Like well-known methods for stable marriages or allocations, we can extend $\Rscr$ to a poset. In this poset rotations $\rho,\rho'\in\Rscr$ are compared as $\rho'\lessdot \rho$ if $\rho'$ is applied \emph{earlier than $\rho$ for all routes} in $\Tscr^0$; clearly relation $\lessdot$ is transitive. The elements $\rho$ are endowed with the weights $\tau(\rho)$, and when needed, we may denote this poset as $(\Rscr,\tau,\lessdot)$.
\medskip

\noindent\textbf{Definition 8.} ~A function $\lambda:\Rscr\to\Rset_+$ satisfying $\lambda(\rho)\le \tau(\rho)$ for all $\rho\in\Rscr$ is called \emph{closed} if  $\lambda(\rho)>0$ and $\rho'\lessdot \rho$  imply $\lambda(\rho')=\tau(\rho')$. A closed function $\lambda$ is called  \emph{fully closed} if for each $\rho\in\Rscr$, the value $\lambda(\rho)$ is either 0 or $\tau(\rho)$. The set of closed functions for $(\Rscr,\tau,\lessdot)$ is denoted by $\frakC=\frakC(\Rscr,\tau,\lessdot)$. It forms a distributive lattice in which for closed $\lambda,\lambda'$, the meet $\lambda\wedge\lambda'$ and join $\lambda\vee\lambda'$ are naturally defined by taking the coordinate-wise minimum and maximum, respectively.
\medskip

Since the maximal weights $\tau(\rho)$ for all $\rho\in\Rscr$ are strictly positive, we observe that for a closed function $\lambda$, both sets $I^+(\lambda):=\{\rho\colon \lambda(\rho)>0\}$ and $I^-(\lambda):=\{\rho\colon \lambda(\rho)=\tau(\rho)\}$ form \emph{ideals} of the poset $(\Rscr,\lessdot)$ (i.e. $\rho\in I^+(\lambda)$ and $\rho'\lessdot \rho$ imply $\rho'\in I^+(\lambda)$, and similarly for $I^-(\lambda)$).

A relationship between the stable assignments and closed functions is established as follows. For $x\in\Sscr$, consider a lower semi-route $T=(\xmin=x_0,x_1,\ldots,x_N=x)$, and for $i=1,\ldots,N$, let $x_i$ be obtained from $x_{i-1}$ by a rotational shift $\lambda_i\rho_i$. By Proposition~\ref{pr:gomogen}, the set $\Omega(T)=\{(\lambda_i,\rho_i)\colon i=1,\ldots,N)\}$ does not depend on a semi-route in $\Tscrlow(x)$. Moreover, the function $\lambda:\Rscr\to \Rset_+$, where $\lambda(\rho)$ is defined to be $\lambda_i$ if $\rho=\rho_i$, $i=1,\ldots,N$, and 0 otherwise, is closed. (Indeed, if $\rho=\rho_i$ and $\rho'\lessdot \rho$, then $\rho'$ must be used in any semi-route in $\Tscrlow(x_i)$, and its generating strong component $L'$ cannot be in $\Lrot(x_i)$, for otherwise $\rho'$ and $\rho$ can be applied in any order. It follows that $\lambda(\rho')=\tau(\rho')$, as required.) We define this closed function by $\omega(x)$.

 \begin{prop} \label{pr:biject}
The correspondence $x\mapsto \omega(x)$ determines a bijection between the lattice $(\Sscr,\prec_W)$ of stable assignments and the lattice $(\frakC,<)$ of closed functions for $(\Rscr,\tau,\lessdot)$.
   \end{prop}
   \begin{proof}
Consider stable assignments $x,y\in\Sscr$. Let $\lambda:=\omega(x)$ and $\mu:=\omega(y)$. The key part of the proof is to show that 
 \begin{equation} \label{eq:homomor}
 \lambda\wedge \mu=\omega(x\curlywedge y)\qquad \mbox{and} \qquad \lambda\vee \mu=\omega(x\curlyvee y)
  \end{equation}
(where $\curlywedge$ and $\curlyvee$ concern the relation $\prec_W$ rather than $\prec_F$); in other words, $\omega$ determines a homomorphism of the lattices. 

To show the left equality, consider a middle (non-expensive) semi-route $X=(z=x_0,x_1,\ldots, x_N)$ from $z:=x\curlywedge y$ to $x$, and a middle semi-route $Y=(z=y_0,y_1,\ldots,y_M)$ from $z$ to $y$. Let $\Rscr(X)=(\rho_1,\ldots,\rho_N)$ and $\Rscr(Y)=(\rho'_1,\ldots,\rho'_M)$ be the corresponding sequences of rotations, and let $\lambda_i$ ($\lambda'_j$) be the corresponding weight of $\rho_i$ (resp. $\rho'_j$) used in $X$ (resp. $Y$). (By Proposition~\ref{pr:gomogen}, the set of rotations $\rho_i$ in $\Rscr(X)$, along with their weights $\lambda_i$, depend only on the end assignments $z$ and $x$. And similarly for $\Rscr(Y)$.) 

One can realize that the left equality in~\refeq{homomor} follows from the property that
  \begin{numitem1} \label{eq:no_common}
the sequences $\Rscr(X)$ and $\Rscr(Y)$ have no common rotations.
  \end{numitem1}
  
We prove~\refeq{no_common} as follows. Suppose, for a contradiction, that $\Rscr(X)$ and $\Rscr(Y)$ share a rotation $\rho$, say, $\rho=\rho_i=\rho'_j$.  

If $i=j=1$ (i.e., $\rho$ is the first rotation in both $\Rscr(X)$ and $\Rscr(Y)$), then $z':=z+\min\{\lambda_1,\lambda'_1\}\,\rho$   satisfies $z'\succ z$, $z'\preceq x$ and $z'\preceq y$, contrary to the fact that $z$ is the maximal lower bound for $x,y$. So we may assume that $i>1$ and $j\ge 1$.

Consider a rotation $\rho_k$ with $k<i$ in $\Rscr(X)$. If $\rho_k\lessdot \rho$, then $\rho_k$ must occur in the sequence $\Rscr(Y)$ earlier than $\rho'_j=\rho$, and therefore, $\rho_k=\rho'_{k'}$ for some $k'<j$. This enables us to exclude the situation when
  \begin{itemize}
\item[($\ast$):] $\rho_k\lessdot \rho_i$ for $k=1,\ldots,i-1$, and simultaneously, $\rho'_{k'}\lessdot \rho'_j$ for $k'=1,\ldots,j-1$.
  \end{itemize}
Indeed, in this case, we would have $\{\rho_1,\ldots,\rho_i\}=\{\rho'_1,\ldots,\rho'_j\}$. Also for each $k=1,\ldots,i-1$, the relation $\rho_k\lessdot \rho$ implies that $\lambda_k$ is the maximal admissible weight $\tau(\rho_k)$. It follows that if $\rho_k=\rho'_{k'}$, then $\lambda_k=\lambda'_{k'}$. Then the stable assignment $z':=z+\lambda_1\rho_1+\cdots+\lambda_{i-1}\rho_{i-1}$ belongs to both semi-routes $X$ and $Y$. This implies $z\prec z'\preceq x$ and $z'\preceq y$, contradicting the definition of $z$.

It remains to consider the situation when there is $\rho_k$ with $1\le k<i$ which is not comparable with $\rho_i$; for brevity, we will write $\rho_i\approx \rho_k$. We will use the following
  \medskip

\noindent\textbf{Claim.} \emph{Let $\tilde x\in\Sscr$ and let $\rho$ and $\rho'$ be rotations such that $\rho$ is applicable to $\tilde x$, and $\rho'$ is applicable to $\tilde x':=\tilde x+\lambda\rho$ for an admissible weight $\lambda>0$ for $\rho$. Then exactly one of the following is valid: {\rm(a)} $\rho'$ is applicable to $\tilde x$, or {\rm(b)} $\rho\lessdot\rho'$.}
  \medskip
  
This Claim will be proved later. Assuming validity of it, we argue as follows. Let $k$ be maximum subject to $k<i$ and $\rho_k\approx\rho_i$. First suppose that $k=i-1$. Then, by (a) in the Claim applied to $\tilde x=x_{i-2}$, the rotations $\rho'=\rho_{i-1}$ and $\rho=\rho_i$ commute, and we can update $X$ by swapping $\rho'$ and $\rho$. This gives the semi-route $X'=(x_0,\ldots,x_{i-2},x',x_i,\ldots,x_N)$, where $x':=x_{i-2}+\lambda_i\rho_i$, and in the corresponding sequence $\Rscr(X')$ the rotation $\rho$ is disposed earlier than in $\Rscr(X)$. We can continue to ``lower'' $\rho$ further if $\rho_{i-2}\approx \rho$, and so on.

Now suppose that $k< i-1$. Then we are going to ``lift'' $\rho'=\rho_k$ so as to get closer to $\rho$. This is done as follows. Since $\rho'\approx \rho$ and $\rho_{k+1}\lessdot\rho$ (by the maximal choice of $k$), we have $\rho'\approx\rho_{k+1}$ (for $\rho'\lessdot\rho_{k+1}$ would imply $\rho'\lessdot \rho$). Then, by~(a) in the Claim, $\rho_{k+1}$ is applicable to $x_{k-1}$. Swapping $\rho'$ and $\rho_{k+1}$ (preserving their weights), we obtain the semi-route $X''$ such that in the corresponding sequence $\Rscr(X'')$ the rotation $\rho'$ becomes closer to $\rho$, as required.

Summing up the above observations, we can conclude that the semi-route $X$ can be transformed so that the rotation $\rho$ reaches the first position in the resulting  sequence of rotations related to $z$ and $x$. By transforming $Y$ in a similar way, we obtain a sequence of rotations related to $z$ and $y$ in which the same rotation $\rho$ is situated at the first position as well. Then we act as in the beginning of the proof, yielding a contradiction with the definition of $z$. 
This proves~\refeq{no_common}, implying the left equality in~\refeq{homomor}. The right equality in~\refeq{homomor} is proved symmetrically (by taking the order $\prec_F$ instead of $\prec_W$, and thinking of a route as a sequence going from $\xmax$ to $\xmin$). Thus, $\omega$ gives a homomorphism of the above lattices, as required.

It follows that the map $\omega$ homomorphically embeds $(S,\prec_W)$ into a sublattice of $(\frakC,<)$. Also $\omega$ is, obviously, injective.

To show that $\omega$ is a bijection, we use the fact that for any $\rho,\rho'\in\Rscr$, unless $\rho'\lessdot \rho$, there is a route where $\rho$ is applied earlier than $\rho'$, and therefore, there is $x\in\Sscr$ separating $\rho$ from $\rho'$, in the sense that $\omega(x)$ assigns a positive weight to $\rho$, but zero weight to $\rho'$. This implies (by using~\refeq{homomor}(left)) that for any $\rho\in\Rscr$ and $0<\lambda\le\tau(\rho)$, there is $x\in\Sscr$ such that $\omega(x)$ assigns weight $\lambda$ to $\rho$, the full weight $\tau(\rho'')$ to each $\rho''\lessdot \rho$, and zero weight otherwise (in other words, $\omega(x)$ is the minimal closed function taking value $\lambda$ on $\rho$). So $\omega(\Sscr)$ covers all join-irreducible elements of $(\frakC,<)$, and the desired bijection follows.
   \end{proof}
   
\noindent\textbf{Remark 4.} ~By Birkhoff's representation theorem~\cite{birk}, a finite distributive lattice $\Lscr$ is isomorphic to the lattice (``ring of sets''), with standard $\cap$ and $\cup$ operations, formed by ideals (down sets) in the poset $\Pscr$ generated by  join-irreducible elements of $\Lscr$. This can be extended, with a due care, to our infinite lattice $\Lscr=(\Sscr,\prec_W)$ as well. Formally, for an integer $k\ge 2$ and any $\rho\in\Rscr$, choose a set $\Lambda^{(k)}(\rho)$ of $k$ weights $0=\lambda_1<\lambda_2<\cdots<\lambda_k=\tau(\rho)$. Accordingly, transform the weighted poset $(\Rscr,\tau,\lessdot)$ into a finite poset $\Pscr^{(k)}$ by replacing each element $\rho$ by $k$ elements $\rho_1,\ldots,\rho_k$ corresponding to the weights in $\Lambda^{(k)}(\rho)$ and ordered as $\rho_1\lessdot\cdots\lessdot\rho_k$, respectively (as to relations $\rho\lessdot\rho'$ in $(\Rscr,\tau,\lessdot)$, they are replaced by $\rho_k\lessdot\rho'_1$). Applying the rotational shifts with chosen weights, we build the corresponding finite subset $\Sscr^{(k)}$  of stable assignments in $\Sscr$. This $\Sscr^{(k)}$ induces a sublattice of $(\Sscr,\prec_W)$ isomorphic to the lattice of ideals in $\Pscr^{(k)}$, in spirit of Birkhoff's representation. (Note that $\Sscr^{(2)}$ consists exactly of the stable assignments generated by full rotational shifts, and $\Pscr^{(2)}$ turns into the unweighted poset $(\Rscr,\lessdot)$.) When $k$ grows, the constructed objects ``tend'' to the original objects $\Sscr,\lessdot_W,\Rscr,\tau,\lessdot$, and one can conclude that the join-irreducible elements $x$ of $(\Sscr,\prec_W)$ are exactly the preimages by $\omega$ of the minimal (viz. \emph{principal}) closed functions determined by the pairs $(\rho,\lambda)$ with $\rho\in\Rscr$ and $0<\lambda\le\tau(\rho)$.   
   
 \medskip  
\noindent\textbf{Proof of the Claim.}
Let $L$ and $L'$ be the maximal strong components in the active graphs $\orar\Gamma=\orar\Gamma(\tilde x)$ and $\orar\Gamma'=\orar\Gamma(\tilde x')$ that determine $\rho$ and $\rho'$, respectively. If $\rho'$ is applicable to $\tilde x$, or, equivalently, $L'\in\Lrot(\tilde x)$, then $\rho\approx \rho'$. So assume that $L'\notin\Lrot(\tilde x)$. Then $\lambda$ is the maximal rotational weight $\tau(\rho)$, whence $\orar\Gamma\ne\orar\Gamma'$. 

For $v\in V^=$, let $A_v$ and $A'_v$ denote the sets of active edges leaving $v$ in $\orar\Gamma$ and $\orar\Gamma'$, respectively. Then (cf.~\refeq{rotat_shift}(b)):
  \begin{numitem1} \label{eq:AAp}
if $v=f\in F^=$, then either $\pic_f(\tilde x')=\pic_f(\tilde x)$ and $A'_f\subseteq A_f$, or $\pic_f(\tilde x')$ is worse than $\pic_f(\tilde x)$; whereas if $v=w\in W^=$, then either $\pic_w(\tilde x')=\pic_w(\tilde x)$ and $A'_w\supseteq A_w$, or $\pic_w(\tilde x')$ is better than $\pic_w(\tilde x)$; moreover, if $\pic_v(\tilde x')=\pic_v(\tilde x)$ for all $v\in V^=$, then at least one of the above inclusions is strict.
  \end{numitem1}
Now consider two possible cases.
 \medskip
 
\noindent\emph{Case 1}: $L$ and $L'$ have a common vertex. For each $v\in V_L\cap V_{L'}$, we have: $A_v=E_L\cap E_v$ and $A'_v=E_{L'}\cap E_v$, the sets $A_v,\,A'_v,\,\pic_v(\tilde x),\,\pic_v(\tilde x')$ satisfy relations as in~\refeq{AAp}, and at least one of these relations does not turn into equality (since, obviously, $L\ne L'$). When moving along any route, relations concerning the active graphs for neighboring assignments are updated similar to those exposed in~\refeq{AAp}. These updates have a ``monotone character'' at each vertex, and we can conclude that the rotation $\rho'$ cannot be applied in any route earlier than $\rho$. It follows that $\rho\lessdot \rho'$.
 \medskip
 
\noindent\emph{Case 2}: $L$ and $L'$ are disjoint. This is possible only if $L'$ is a strong component in each of $\orar\Gamma$ and $\orar\Gamma'$, but $L'$ is maximal in $\orar\Gamma'$ and not maximal in $\orar\Gamma$. Then the set $Q$ of active edges in $\orar\Gamma$ going from $L'$ to other strong components in $\orar\Gamma$ must disappear under the transformation of $\tilde x$ into $\tilde x'$. This implies that all edges in $Q$ enter just $L$. Since the assignment on each $e\in Q$ is not changed: $\tilde x'(e)=\tilde x(e)$, the situation that the edge $e$ becomes non-active can happen only if $e$ goes from $F$ to $W$, say, $e=(f,w)$, and $w$ does not belong to the tail $T_w(\tilde x')$. This implies that the assignments on the edges in $\pic_w(\tilde x)$ must reduce to zero. Then the critical tie in $\Pi_w$ becomes better: $\pic_w(\tilde x')>\pic_w(\tilde x)$. Moreover, $e$ must belong to a tie worse than $\pic_w(\tilde x')$. It follows that in any route the rotation $\rho'$ cannot be applied earlier than $\rho$ (for otherwise the edges from $Q$ should be active at the moment of applying $\rho'$, and therefore, $L'$ could not be a maximal strong component in the current active graph). As a result, we obtain $\rho\lessdot \rho'$.
\hfill$\qed$
\medskip   
   
In the rest of this section, we explain how to construct the rotation poset $(\Rscr,\tau,\lessdot)$ efficiently. Our approach is in spirit of the method elaborated for the stable marriage problem in~\cite{IL}.

First of all we compute, step by step, the set $\Rscr$ of rotations, see explanations around Proposition~\ref{pr:time}. This is performed in the process of constructing an arbitrary route $T$ from $\xmin$ to $\xmax$. Recall that at each step, for a current stable assignment $x$, we construct the active graph $\orar\Gamma(x)$, extract the set $\Lrot(x)$ of maximal strong components $L$ in it, and determine the rotations $\rho=\rho_L$ and their maximal admissible weights $\tau(\rho)=\tau_L$ (this is performed in time polynomial in $|E|$, cf. Proposition~\ref{pr:time}). Then we update the current $x$ by adding to it the total shift $\sum(\tau_L\rho_L\colon L\in\Lrot(x))$, and so on. As a result, we obtain the set $\Rscr$ of all  rotations $\rho$ (where $|\Rscr|=O(|E|)$) along with their maximal weights $\tau(\rho)$, in time polynomial in $|E|$.

A less trivial task is to find the order relations $\lessdot$ for $\Rscr$. This is reduced to finding the set of pairs $(\rho,\rho')$ giving \emph{immediately preceding} relations, i.e., those satisfying $\rho\lessdot \rho'$ and having no $\rho''\in\Rscr$ in between: $\rho\lessdot\rho''\lessdot\rho'$. We try to extract such pairs by appealing to the Claim in the proof of Proposition~\ref{pr:biject}.

More precisely, define $\Escr$ to be the set of pairs $(\rho,\rho')$ in $\Rscr$ such that
  \begin{numitem1} \label{eq:rho-rhop}
there exists $x\in\Sscr$ such that $\rho$ but not $\rho'$ is applicable to $x$, whereas $\rho'$ is applicable to $x':=x+\tau(\rho)\rho$.
  \end{numitem1}
  
Each pair $(\rho,\rho')\in\Escr$ gives an immediately preceding relation. (Indeed, $x$ and $x'$ as above can be represented as two consecutive elements in some route; also the Claim implies that $\rho\lessdot \rho'$; therefore, the existence of $\rho''$ between $\rho$ and $\rho'$ is impossible.)

Note that the set of pairs $\Escr$ can be found efficiently (in fact, in time $O(|E|^2)$). 

We show that $\Escr$ does determine all immediately preceding relations in $(\Rscr,\lessdot)$ (in other words, $(\Rscr,\Escr)$ is the Hasse diagram for this poset). Clearly if $\rho'\in\Rscr$ is reachable from $\rho\in\Rscr$ by a directed path in $(\Rscr,\Escr)$, then $\rho\lessdot \rho'$. The converse is valid as well.
 \begin{prop} \label{pr:dir_path}
If $\rho,\rho'\in\Rscr$ and $\rho\lessdot\rho'$, then the graph $\Hscr=(\Rscr,\Escr)$ has a directed path from $\rho$ to $\rho'$.
 \end{prop}
 \begin{proof}
~Let $\Iscr$ be the set of ideals in the poset $(\Rscr,\lessdot)$, and $\Jscr$ the set of ideals in the poset determined by $\Hscr$ (via the reachability by paths). Then $\Iscr\subseteq \Jscr$, and we have to show that $\Iscr=\Jscr$, whence the result will follow.

Suppose this is not so. Choose $J\in \Jscr$ with $|J|$ maximum subject to $J\notin \Iscr$. Also choose $R\in\Iscr$ with $|R|$ maximum subject to $R\subset J$. By Proposition~\ref{pr:biject}, $R$ corresponds to some $x\in\Sscr$, and therefore, $R$ is the set of rotations used to form a lower semi-route $T$ from $\xmin$ to $x$ (i.e., $R=\Rscr(T)$).

Consider a rotation $\rho$ applicable to $x$. Then $R':=R\cup\{\rho\}$ is an ideal in $\Iscr$ (corresponding to the assignment $x':=x+\tau(\rho)\rho$). 
By the maximality of $R$, we have $R'\not\subseteq J$; so $\rho$ is not in $J$. Then the set $J':=J\cup R'=J\cup\{\rho\}$ is an ideal in $\Jscr$ strictly including $J$, and the maximality of $J$ implies $J'\in\Iscr$.

Now since $J'\supset R'$ and both $J',R'$ are ideals in $\Iscr$, the set $J'-R'$ contains a rotation $\rho'$ such that the set $R'':=R'\cup\{\rho'\}$ is an ideal in $\Iscr$. Form $J'':= R\cup\{\rho'\}$; obviously, $R\subset J''\subseteq J$. This $J''$ is not an ideal in $\Iscr$ (since $J''\in\Iscr$ would lead to a contradiction with the maximality of $R$ (if $J''\ne J$) or with the condition on $J$ (if $J''=J$)).

So we come to the situation when $R,R',R''\in\Iscr$; $R'=R\cup\{\rho\}$; $R''=R'\cup\{\rho'\}$; but $R\cup\{\rho'\}\notin\Iscr$. Then the rotation $\rho'$ is not applicable to the assignment $x$, but it is applicable to $x'$ (obtained from $x$ by applying $\rho$). This means that $(\rho,\rho')$ is an edge in $\Escr$ (cf.~\refeq{rho-rhop}). But this edge goes from $\Rscr-J$ to $J$ (since $\rho\notin J$ while $\rho'\in J$), contrary to the fact that $J$ is an ideal for $\Hscr$. 
 \end{proof}
 
\section{Affine representability and stable assignments of minimum cost} \label{sec:weight}

The bijection $\omega$ figured in Proposition~\ref{pr:biject} enables us to give an affine representation of the lattice $\Lscr=(\Sscr,\prec_W)$, like what is noticed in~\cite{MS} for SAP, or in~\cite{FZ} for the Boolean case of Alkan--Gale's model.

In what follows, depending on the context, a rotation $\rho\in\Rscr$ can be regarded either as the corresponding function on $E$ (viz. a vector in $\Rset^E$), or as a unit base vector in the space $\Rset^\Rscr$ (of which coordinates are indexed by elements of $\Rscr$); in the latter case, we may use notation $\langle\rho\rangle$ for $\rho$. Let $A\in \Rset^{E\times\Rscr}$ be the matrix whose columns are formed by rotations $\rho\in\Rscr$ (regarded as vectors in $\Rset^E$). Then, by Proposition~\ref{pr:biject},
  \begin{numitem1} \label{eq:affine}
the affine map $\gamma:\lambda\in\Rset^\Rscr\mapsto x\in\Rset^E$ given by
$x:= \xmin+A\lambda$
determines the bijection between the lattice $(\frakC,<)$ of closed functions $\lambda$ for the rotation poset $(\Rscr,\tau,\lessdot)$ and the lattice $(\Sscr,\prec_W)$ of stable assignments in SMP; in other words, $x=\omega^{-1}(\lambda)$ (where $x\in\Sscr$, $\lambda\in\frakC$, and $\omega$ is as in Proposition~\ref{pr:biject}).
  \end{numitem1}
  
As is explained in Sects.~\SEC{addit_prop} and~\SEC{poset_rot}, $|\Rscr|$ is estimated as $O(|E|)$ and the matrix $A$ can be constructed in time polynomial in $|E|$. Due to this, we can take advantages from the above affine representation in order to efficiently translate certain problems on stable assignments into problems on closed functions. 

First of all we observe that the map $\gamma$ in~\refeq{affine} sends the convex hull $\convex(\frakC)\subset\Rset^\Rscr$ of $\frakC$ onto the convex hull $\convex(\Sscr)\subset\Rset^E$ of $\Sscr$. The polytope $\convex(\frakC)$ is full-dimensional in $\Rset^\Rscr$ (since each unit base vector $\langle\rho\rangle$ can be expressed as a linear combination of closed functions) and it can be described by the linear constraints
  \begin{eqnarray}
  0\le\lambda(\rho)\le\tau(\rho),\qquad\quad && \rho\in\Rscr;  \label{eq:conv72} \\
  (1/\tau(\rho))\lambda(\rho)\le (1/\tau(\rho'))\lambda(\rho'), && \rho,\rho'\in\Rscr,\;
                \rho'\lessdot\rho. \label{eq:conv73}
   \end{eqnarray}                
(Note that~\refeq{conv72}--\refeq{conv73} turn into the constraints for Stanley's order polytope~\cite{stan} after scaling  each variable $\lambda(\rho)$ to $\lambda(\rho)/\tau(\rho)$.) 

When $A$ has full column rank (and therefore, $\dimen(\convex(\Sscr))=  \dimen(\convex(\frakC))=|\Rscr|$), $\convex(\Sscr)$ becomes affinely congruent to $\convex(\frakC)$, and the facets of $\convex(\Sscr)$ can be explicitly described via the facets of $\convex(\frakC)$; the latter ones are as in~\refeq{conv72} for minimal  $\rho$ in $(\Rscr,\lessdot)$, and as in~\refeq{conv73} when $\rho'$ immediately precedes $\rho$. This is valid, e.g., for the Boolean case of~\cite{AG}, as is shown in~\cite{FZ}, and for the unconstrained case of SAP (when $b\equiv\infty)$), as is shown in~\cite{MS}. As to our model SMP, the full column rank is not guaranteed even in the unconstrained case $b\equiv\infty$ (to construct a counterexample is not difficult). When $k:=|\Rscr|-\dimen(\convex(\Sscr))>0$, the facets of $\convex(\Sscr)$ are images by $\gamma$ of some faces (or their unions) of co-dimension $k+1$ in $\convex(\frakC)$, and to characterize these facets looks a less trivial task (for a further discussion, see Remark~5 below). 

Next, note that
\begin{numitem1} \label{eq:vertC}
the vertices of $\convex(\frakC)$ are exactly the fully closed functions.
  \end{numitem1}

Indeed, if $\lambda\in\frakC$ is such that $0<\lambda(\rho)<\tau(\rho)$ holds for some $\rho\in\Rscr$, then $\lambda$ is a convex combination of closed functions $\lambda'$ and $\lambda''$ obtained from $\lambda$ by replacing its value on $\rho$ by zero and by $\tau(\rho)$, respectively; so $\lambda$ is not a vertex. Conversely, any fully closed function $\lambda$ is a vertex of $\convex(\frakC)$ since it is the unique element of $\frakC$ attaining the maximum of the inner product $a\lambda'$ over $\lambda'\in\frakC$, where for $\rho\in\Rscr$,  $a(\rho)$ is defined to be 1 if $\lambda(\rho)=\tau(\rho)$, and $-1$ if $\lambda(\rho)=0$. 

We obtain from~\refeq{affine} and~\refeq{vertC} that
  \begin{numitem1} \label{eq:vertS}
if $x\in\Sscr$ is a vertex of $\convex(\Sscr)$, then $\omega(x)$ is fully closed (but the converse is not guaranteed when $\gamma$ is singular).
  \end{numitem1}

In the rest of this section we consider the \emph{min-cost stable assignment problem}:
  \begin{numitem1} \label{eq:min_cost}
given \emph{costs} $c(e)\in\Rset$ of edges $e\in E$, find a stable assignment $x\in\Sscr$ minimizing the total cost $cx:=\sum(c(e)x(e)\colon e\in E)$.
 \end{numitem1} 
(Since all stable assignments $x$ have the same size $|x|$, the function $c$ can be considered up to adding a constant; in particular, one may assume that $c$ is positive. Replacing $c$ by $-c$ leads to the problem of maximizing $cx$.)

Starting with the pioneering work~\cite{ILG}, a well-known method of minimizing a linear function over the set of stable objects (in an appropriate stability model) consists in reduction to a minimization problem over the set of ideals or closed functions in the associated poset of rotations (assuming that such a poset exists and can be efficiently constructed). The latter problem is further reduced, by use of Picard's method in~\cite{pic}, to the classical minimum cut problem. To make our description more self-contained, we now explain how this approach works in our case. 

We first compute the cost of each rotation $\rho\in\Rscr$, to be $c\rho:=\sum(c(e)\rho(e)\colon e\in E)$. For each $x\in\Sscr$ and the corresponding closed function $\lambda=\omega(x)$, we have
  $$
  cx=c\xmin+\sum(\lambda(\rho)c\rho\colon \rho\in\Rscr).
  $$

Since problem~\refeq{min_cost} has an optimal solution $x\in\Sscr$ for which $\omega(x)$ is fully closed, by~\refeq{vertS}, we come to the following minimization problem over the set of ideals in the finite poset $(\Rscr,\lessdot)$, with the weight function $\zeta$ defined by $\zeta(\rho):=\tau(\rho)c\rho$ for $\rho\in\Rscr$:
  \begin{numitem1} \label{eq:zeta_cost}
find an ideal $I\subset \Rscr$ minimizing the total weight $\zeta(I):=\sum(\zeta(\rho)\colon \rho\in I)$. 
\end{numitem1}

Admitting a slightly more general setting (\emph{problem}~$(\ast)$), we can consider a finite directed graph $H=(V_H,E_H)$ and a function $\zeta:V_H\to\Rset$ and ask for finding a closed set of vertices $X\subseteq V_H$ minimizing $\zeta(X):=\sum(\zeta(v)\colon v\in X)$ (where $X$ is called closed if there is no edge going from $V_H-X$ to $X$).

Following~\cite{pic}, problem~$(\ast)$ is reduced to an instance of minimum cut problem on the directed graph $\hat H=(\hat V,\hat E)$ and edge capacity function $h$ obtained from $H,\zeta$ by: 

(a) adding two vertices, ``source'' $s$ and ``sink'' $t$; 

(b) adding the set $E^+$ of edges $(s,v)$ for all vertices $v$ in $V^+:=\{v\in V_H\colon \zeta(v)>0\}$;

(c) adding the set $E^-$ of edges $(u,t)$ for all vertices $u$ in $V^-:=\{u\in V_H\colon \zeta(u)<0\}$;

(d) assigning $h(s,v):=\zeta(v)$ for all $v\in V^+$, \; $h(u,t):=|\zeta(u)|$ for all $u\in V^-$, and $h(e):=\infty$ for all $e\in E_H$.

An $s$--$t$ \emph{cut} in $H$ is meant to be the set $\delta(A)$ of edges going from a set $A\subset \hat V$ such that $s\in A\not\ni t$ to its complement $\hat V-A$, and its capacity is $h(\delta(A)):=\sum(h(e)\colon e\in\delta(A))$. One shows that $\delta(A)$ is a minimum capacity cut among all $s$--$t$ cuts if and only if $X:=V_H-A$ is a closed set in $H$ with $\zeta(X)$ minimum.

Indeed,  for an $s$--$t$ cut $\delta(A)$, the value $h(\delta(A))$ is finite if and only if $\delta(A)$ contains no edges from $H$. The latter implies that $\delta(A)\subseteq E^+\cup E^-$, and $X$ is closed. Then
   \begin{multline*}
h(\delta(A))=h(\delta(A)\cap E^+)+h(\delta(A)\cap E^-) \\
= \zeta(X\cap V^+)+\sum(|\zeta(u)|\colon u\in (V_H-X)\cap V^-)\\
  =\zeta(X\cap V^+)+ \zeta(X\cap V^-)-\zeta(V^-)=\zeta(X)-\zeta(V^-).
  \end{multline*}

So $\zeta(X)$ differs from $h(\delta(A))$ by the constant $\zeta(V^-)$, and the desired property follows. Summing up the above reasonings and constructions, we obtain
 \begin{theorem} \label{tm:min_cost}
The minimum cost stable assignment problem~\refeq{min_cost} is solvable in strongly polynomial time.
  \end{theorem}
  
\noindent\textbf{Remark 5.} As is mentioned in the Introduction, the (strongly) polynomial-time solvability of~\refeq{min_cost} implies that the separation problem for $\Sscr$ is solvable in (strongly) polynomial time as well, due to the polynomial-time equivalence between optimization and separation problems (in a wide class of such problems) established in~\cite[Theorems~6.4.9,\,6.6.5]{GLS}. (Recall that the separation problem for $\Sscr$ is: given a vector $\xi\in\Rset^E$, decide whether $\xi$ is contained in $\convex(\Sscr)$, and if not, find a hyperplane $H$ in $\Rset^E$ separating $\xi$ and $\Sscr$, i.e. such that one component (half-space) in $\Rset^E-H$ contains $\xi$ but no element of $\Sscr$.) Moreover, when $\xi\notin\convex(\Sscr)$, one can find, in (strongly) polynomial time, a separating hyperplane including a facet of $\convex(\Sscr)$ (cf.~\cite[Th.~6.5.16]{GLS}).

In light of this, one may ask: (a) can $\convex(\Sscr)$ be explicitly described via a system of linear constraints, or~(b) can one characterize the facet structure of $\convex(\Sscr)$ in ``good combinatorial terms''? Both questions are affirmatively answered in the simplest case, the stable marriage problem, in~\cite{VV}. An affirmative answer to~(a) is given in~\cite{MS} for the unconstrained SAP, and in~\cite{FZ} for the Boolean version of~\cite{AG} in the case when at least one side of the market (viz. vertex part of $G$) is subject to quota restrictions and the choice functions are given via oracle calls. The latter generalizes known many-to-many stability problems. As to our SMP model, both questions are open at present.

\appendix

\section{Appendix. An efficient algorithm for finding $\xmin$}

Recall that the method described in Sect.~\SEC{begin} can have infinitely many iterations, though it converges just to the stable assignment $\xmin$ which is optimal for the part $F$. We refer to that method, as well as iterations in it, as \emph{ordinary}. In this section we describe a modified method, which is finite and finds $\xmin$ in time polynomial in $|E|$.

We start with analyzing some details of the ordinary method. Recall that a current, $i$-th, iteration in it consists of two stages: at the 1st one, the current boundary function $b^i$ is transformed into $x^i$, by setting $x^i_f:=C_f(b^i_f)$ for each $f\in F$, and at the 2nd one, $x^i$ is transformed into $y^i$, by setting $y^i_w:=C_w(x^i_w)$ for each $w\in W$. Then the boundary function is updated according to rule~\refeq{bi}. Let $F^{i=}$ denote the set of fully filled vertices $f$ in $F$ for $x^i$ (i.e., satisfying $|x^i_f|=q(f)$), and $W^{i=}$ the set of fully filled vertices in $W$ for $y^i$. For $f\in F$, let $L^i_f$ denote the set of edges $e\in E_f$ such that $b^i(e)<b(e)$. We observe the following.
  \smallskip
  
(i) For each $f\in F$, the sets $L_f$ are monotone non-decreasing: $L^i_f\subseteq L^{i+1}_f$, and the current value on each edge $e\in L_f$ may decrease (at the 2nd stages of some iterations) but never increases (since $b^1(e)\ge \cdots \ge b^i(e)\ge b^{i+1}(e)\ge\cdots$).
 \smallskip
 
(ii) The sets $W^{\bullet=}$ are monotone non-decreasing: $W^{i=}\subseteq W^{(i+1)=}$ (for if $|x^i_w|\ge q(w)$, then the choice function $C_w$ cuts $x^i_w$ so as to get $|y^i_w|=q(w)$, and subsequently the relations $|x_w|\ge|y_w|=q(w)$ are always maintained). In turn, the sets $F^{\bullet=}$ are monotone non-increasing: $F^{i=}\supseteq F^{(i+1)=}$ (since $b^i\ge b^{i+1}$).
 \smallskip
 
(iii) The heads $H_w(y)$ of vertices $w\in W^{\bullet=}$ are monotone non-decreasing: $H_w(y^i)\subseteq H_w(y^{i+1})$ (when $w\in W^{i=}$), unless the critical tie in $\Pi_w$ changes (becomes better). In turn, the heads $H_j(x)$ for $f\in F^{\bullet=}$ are monotone non-increasing: $H_f(x^i)\supseteq H_f(x^{i+1})$.
 \smallskip
 
(iv) For each $f\in F$, all edges in $E_f-(H^i_f\cup L^i_f)$ are saturated, where $H^i_f$ denotes the head $H_f(x^i)$ (which is nonempty if $f\in F^{i=}$, and empty if $f\in F-F^{i=}$).
\smallskip
 
These observations prompt monotone parameters, and we say that $i$-th iteration is \emph{positive} if at least one of the following events happens in it: for some $f\in F$, the set of edges $e\in E_f$ that are either saturated or contained in $L_f$ increases; the set $W^{\bullet=}$ of fully filled vertices in $W$ increases; for some $w\in W^{\bullet=}$, either the head $H_w$ increases, or the critical tie for $w$ becomes better. The number of positive iterations does not exceed $2|E|+|W|$. 

Next we deal with a sequence $Q$ of consecutive non-positive iterations and consider a current iteration among these. At the 1st stage of this iteration, the sum of values of the current function $x$ increases (compared with the preceding $y$) by some amount $\xi>0$. Note that if for a vertex $w\in W^{\bullet=}$, the value increases on some edge $e=fw$, then this edge cannot belong to the head $H_w$ or to $\Eend_w$. (For otherwise the cutting operation at the 2nd stage of the current iteration decreases the value on $e$, and as a consequence, $e$ is added as a new element to the set $L_f$, which is impossible since the iteration is non-positive.)

It follows that an increase on some edge $e=fw$ at the 1st stage of an iteration from $Q$ is possible only in two cases: either $w\in W^{\bullet=}$ and $e$ is contained in the current tail $T_w$, or $w$ belongs to the current deficit set $W-W^{\bullet=}$. This implies that the total increase $\xi'$ on the edges incident to vertices in $W^{\bullet=}$ satisfies $\xi'\le\xi$ (the equality may happen only if there is no contribution from $\xi$ to the vertices in $W-W^{\bullet=}$). At the 2nd stage of this iteration, the total value in the heads $H_w$ for these vertices $w\in W^{\bullet=}$ must decrease by the same amount $\xi'$ (emphasize that these heads $H_w$ do not change, and each edge $fw$ where the value decreases has been added to the set $L_f$ on earlier iterations). Then we come to the 1st stage of the next iteration with the total ``quota-underloading'' $\xi'$ in some vertices in $F$, which further causes a ``quota-overloading'' $\xi''\le\xi'$ in $W^{\bullet=}$, and so on. When such ``overloadings'' decrease geometrically, non-positive iterations may continue infinitely long (yet providing a convergence to a limit function).

Relying on the observations above, we are now going to aggregate iterations from $Q$. More precisely, consider the first iteration in $Q$; assume that it is $i$-th iteration in the ordinary process. For brevity, we denote $x^i$ by $x$, and $y^i$ by $y$ (the functions obtained on the 1st and 2nd stages of the first iteration in $Q$). Also denote:
  %
  \begin{gather*}
 F^= :=F^{i=}\quad \mbox{and}\quad W^= :=W^{i=}; \\
  H_f:=H_f(x),\;\; h_f:=|H_f|,\;\; \mbox{and}\;\; a(y,f):=y(e) 
   \quad\mbox{for $f\in F^=$ and $e\in H_f$}; \\  
  H_w:=H_w(y),\;\; h_w:=|H_w|,\;\; \mbox{and}\;\; a(y,w):=y(e) 
   \quad\mbox{for $w\in W^=$ and $e\in H_w$}. 
  \end{gather*}

The  sequence $Q$ without the first iteration in it is replaced by one ``big'' (aggregated) iteration. To do so, we arrange the following linear program with variables $\phi_f$ for $f\in F^=$, and $\psi_w$ for $w\in W^=$:
  \begin{gather}
  \mbox{maximize\;\; $\eta(\phi):=\sum(h_f\phi_f\colon f\in F^=)$ \qquad\quad subject to}
     \label{eq:Alp1}  \\
  h_w\psi_w-\sum(\phi_f\colon f\in F^=,\, fw\in H_f)=0,\qquad w\in W^=;\label{eq:Alp2} \\  
  \qquad\qquad \qquad \psi_w\le a(y,w), \qquad\qquad\qquad\qquad w\in W^=; \label{eq:Alp3} \\
  h_f\phi_f-\sum(\psi_w\colon w\in W^=,\, fw\in L_f)\le q(f)-|y_f|,
           \quad f\in F^=;                           \label{eq:Alp4}\\
\qquad\qquad\qquad\qquad    \phi_f\le b(fw)-y(fw), \qquad\qquad\qquad f\in F^=,\; fw\in H_f; \label{eq:Alp5} \\
\qquad    \phi_f+\psi_w\le a(y,w)- y(fw), \qquad f\in F^=,\, w\in W^=,\, fw\in 
               \pic_w(y)-H_w; \label{eq:Alp6} \\
    \qquad  \sum(\phi_f\colon f\in F^=,\, fw\in E_w)\le q(w)-|y_w|, \qquad\qquad w\in W-W^=;
                  \label{eq:Alp7} \\    
   \phi\ge 0,\; \psi\ge 0. \qquad\label{eq:Alp8}
   \end{gather}

The set of feasible solutions to this system is nonempty (since we can take as $\phi,\psi$ the corresponding augmentations determined on the second iteration in $Q$) and compact (in view of~\refeq{Alp3},\refeq{Alp5},\refeq{Alp8}). Let $(\phi,\psi)$ be an optimal solution to~\refeq{Alp1}--\refeq{Alp8}, and let $\Delta=\Delta^{\phi,\psi}$ be the corresponding function on $E$, i.e., for $e=fw\in E$, $\Delta(e)$ is equal to $\phi_f$ if $f\in F^=$ and $e\in H_f$,\; $-\psi_w$ if $w\in W^=$ and $e\in H_w$, and 0 otherwise. 

Define $y':=y+\Delta^{\phi,\psi}$. Then, based on the properties of $y$, we observe the following: \refeq{Alp2} implies that all vertices in $W^=$ are fully filled for $y'$; \refeq{Alp3} implies that $y'$ is nonnegative; \refeq{Alp4} implies that $y'$ does not violate the quotas within $F^=$; \refeq{Alp5} implies that $y'$ does not exceed the upper bounds $b$; \refeq{Alp6} can be rewritten as $y'(e)=a(y,w)-\psi_w\ge y(fw)+\phi_f=y'(fw)$, where $e\in H_w$, implying that the edges in the head $H_w$ for $y$ belong to the head at $w$ for $y'$ as well; and~\refeq{Alp7} implies that $y'$ does not violate the quotas within $W-W^=$. 

One may assume that the optimal solution $(\phi,\psi)$ is chosen so that the set $\Lscr$ of inequalities in~\refeq{Alp3}--\refeq{Alp7} that turn into equalities is nonempty. One can see that if $\Lscr$ involves a constraint from~\refeq{Alp3},\refeq{Alp5}--\refeq{Alp7}, then some of the above-mentioned parameters changes. (Indeed, the equality with $w$ in~\refeq{Alp3} implies that $y'$ becomes zero on $H_w(y)$, the equality with $fw$ in~\refeq{Alp5} implies that $y'(fw)$ attains the upper bound $b(fw)$; the equality with $fw$ in~\refeq{Alp6} implies that the head at $w$ becomes larger; and the equality with $w$ in~\refeq{Alp7} implies that the vertex $w$ becomes fully filled.)

Suppose that none of the constraints in~\refeq{Alp3},\refeq{Alp5}--\refeq{Alp7} turns into equality, and therefore, the attained constraints concern only~\refeq{Alp4}. Then the resulting function $y'$ is well-defined, in the sense that it is suitable to applying further ordinary iterations. Let $x''$ be the function obtained by applying to $y'$ the 1st stage of the ordinary iteration. We observe that if $\Lscr$ involves all inequalities in~\refeq{Alp4}, then all vertices in $F^=$ are fully filled for $y'$; this implies that $x''=y'$, and the process terminates.

Now suppose that there is an inequality in~\refeq{Alp4} that is strict for $y'$, and therefore, the subset $F'$ formed by $f\in F^=$ with $|y'_f|<q(f)$ is nonempty. Let $y''$ be the functions obtained by applying to $x''$ the 2nd stage of the ordinary iteration. Then $\delta:=y''-y'$ is represented as $\Delta^{\phi',\psi'}$, where $\phi'$ is a function on $F^=$ such that $\phi'_f>0$ for each $f\in F'$. It follows that the pair $(\phi'',\psi''):=(\phi+\phi',\psi+\psi')$ gives a feasible solution to~\refeq{Alp1}--\refeq{Alp8}. But $\eta(\phi'')=\eta(\phi)+\eta(\phi')> \eta(\phi)$, contrary to the optimality of $(\phi,\psi)$.

Thus, the big iteration either terminates the whole process or changes at least one of the above parameters. Also in the modified method as described above, each big iteration is preceded by one ordinary iteration and there are no two consecutive ordinary non-positive iterations. Therefore, the whole process finishes after $O(|E|)$ iterations. Each big iteration is performed in time polynomial in $|E|$, by using a strongly polynomial l.p. method (see, e.g.,~\cite{schr}) to solve program~\refeq{Alp1}--\refeq{Alp8}. 

Let $\hat x$ be the resulting assignment. To see that $\hat x$ is stable, one can argue in a similar way as in the ordinary method (where a proof of stability is rather straightforward). More precisely, one can show by induction on the number of iterations that if an edge $e=fw$ is not saturated upon termination of a current iteration, then $e$ belongs to either the set $H_f\cup\Eend_f$ for $f\in F^=$, or to the set $H_w\cup\Eend_w$ for $w\in W^=$, whence the stability of $\hat x$ will follow. (Here $W^=,H_w,\Eend_w$ are objects arising at the end of the iteration, whereas $F^=,H_f,\Eend_f$ are those arising on the 1st stage when the iteration is ordinary, and are defined as above (before introducing system~\refeq{Alp1}--\refeq{Alp8})  when the iteration is big.) 

It remains to ask: whether the resulting $\hat x$ found by the modified algorithm does coincide with the $F$-optimal assignment $\xmin$? In essence this question is not so important to us, because once we are given one or another stable assignment, we can always transform it into $\xmin$ by use of a sequence of $O(|E|)$ rotational shifts, by considering the polar order $\succeq_W$ on $\Sscr$ and acting as in Sect.~\SEC{poset_rot}. 
As a result, we can conclude with the following
 \begin{prop} \label{pr:modif}
The above modified algorithm to find $\xmin$ is finite and consists of $O(|E|)$ iterations, each taking time polynomial in $|E|$.    \hfill$\qed$
  \end{prop}
  
\noindent\textbf{Remark 6.} ~As a (partial) alternative to the above method, we can offer an efficient and relatively simple ``combinatorial'' algorithm which, given an instance of SMP, decides whether a stable solution to it is quota filling (i.e. any stable solution attains the quotas for all vertices), and if does, finds a stable solution. To do so, the input graph $G=(F\sqcup W,E)$ is extended to graph $G'=(F'\sqcup W',E')$ by adding:

vertex $f_0$ (to $F$) and vertex $w_0$ (to $W$);

edges $f_0w$ for all $w\in W$, denoting the set of these edges by $A$;

edges $fw_0$ for all $f\in F$, denoting the set of these edges by $B$;

edge $f_0w_0$.

The capacities $b$ are extended to $E'-E$ as
  \begin{numitem1} \label{eq:Ep-E}
~$b(f_0w):=q(w)$, $w\in W$;\;\; $b(fw_0):=q(f)$, $f\in F$;\;\; and $b(f_0w_0):=\infty$,
  \end{numitem1}
and the quotas are extended as 
  \begin{numitem1} \label{eq:quotVp}
~$q(f_0):=\sum(q(w)\colon w\in W)=:Q_W$\; and\; $q(w_0):=\sum(q(f)\colon f\in F=:Q_F$.
  \end{numitem1}

The preferences are extended as follows (keeping the old orders $\ge_v$ within $G$):
  \begin{numitem1} \label{eq:pref_new}
\begin{itemize}
 \item[(a)] for $f\in F$ and all  $e\in E_f$, assign $fw_0\succ_f e$;
 \item[(b)] for $w\in W$ and all $e\in E_w$, assign $e\succ_w f_0w$;
 \item[(c)] assign $e\succ_{f_0} f_0w_0$ for all $e\in A$,\; and assign $f_0w_0\succ_{w_0} e'$ for all $e'\in B$;
  \item[(d)] assign arbitrary preferences on $A$ w.r.t. $f_0$, and on $B$ w.r.t. $w_0$.
   \end{itemize}
   \end{numitem1}

For convenience we will use notation with symbol $y$ for assignments on $E'$, and with $x$ for the restrictions of $y$ on $E$. Define $y_0$ on $E'$ as

  \begin{numitem1} \label{eq:y0}
~$y_0(e):=b(e)$ for all $e\in A$,\;  $y_0(e'):=b(e')$ for all $w\in W$, and 0 otherwise.
  \end{numitem1}
This $y_0$ is quota filling (in $G'$) and stable. Moreover~\refeq{pref_new}(a) ensures that $y_0$ is the best assignment for $F'$, i.e. $y_0=\ymin$. One can show the following
  \begin{lemma} \label{lm:GGp}
{\rm (i)} Let $x$ be a stable assignment for $G$ and suppose that $x$ is quota filling (which implies that $Q_F=Q_W$). Let $y$ be the extension of $x$ to $E'$ that takes zero values on $A\cup B$ and satisfies $y(f_0w_0)=Q_F$. Then $y$ is a stable assignment for $G'$.

{\rm (ii)} Let $y$ be a stable assignment for $G'$ and suppose that $y$ takes zero values on $A\cup B$. Then $x=y\rest{E}$ is a quota filling stable assignment for $G$.
 \end{lemma}
   \begin{proof}
~(i)  follows from the facts that: $y$ is quota filling (since $y(f_0w_0)=q(f_0)=q(w_0)$); and no edge in $E'-E=A\cup B\cup\{f_0w_0\}$ is blocking for $y$ (since (a),(b),(c) in~\refeq{pref_new} imply that: for each $w\in W$, $f_0w$ is the worst in $E'_w$; for each $f\in F$, $fw_0$ is worse than the unique edge $f_0w_0$ in $E'_{w_0}$ where $y$ is nonzero; and $f_0w_0$ is the worst in $E'_{f_0}$).

Now we show~(ii). Observe that any $w\in W$ is fully filled by $x$. For otherwise $w$ is deficit for $y$ (since $y(E'_w)=x(E_w)+y(f_0w)=x(E_w)<q(w)$), implying that the edge $e=f_0w$ is blocking for $y$ in both cases: $y(f_0w_0)<q(f_0)$, and $y(f_0w_0)=q(f_0)$ (in view of $e\succ_{f_0} f_0w_0$ and $q(f_0)=Q_W>0$). So we have $|x_w|=q(w)$ for all $w\in W$, and therefore,
  $$
  Q_W=\sum(x(E_w)\colon w\in W)=\sum(x(E_f)\colon f\in F)\le Q_F.
  $$

The case $Q_W<Q_F$ is impossible. For otherwise we would have $y(E'_{w_0})=y(f_0w_0)<q(w_0)=Q_F$ and $y(E'_f)=x(E_f)<q(f)$ for some $f\in F$, and therefore, the edge $fw_0$ would be blocking for $y$ (as it is unsaturated and connects two deficit vertices). Thus, $Q_W=Q_F$. This implies that $x$ is quota filling. Now the stability of $x$ easily follows.
   \end{proof}

Based on Lemma~\ref{lm:GGp}, we apply to $G'$ the rotational method to obtain $\ymax$, by drawing a route from $\ymin$ to $\ymax$. Let $x$ be the restriction of $y=\ymax$ to $E$. If $y$ takes zero values on $A\cup B$, then, by~(ii) in the lemma, $x$ is stable and quota filling. This implies that $x=\xmax$ (since $\xmax$ admits a proper extension to $G'$, according to~(i) in the lemma). And if $y$ is not identically zero on $A\cup B$, then we can declare that the stable assignments for $G$ are not quota filling; this again follows from~(i) in the lemma. 

Thus, we obtain an efficient algorithm, not appealing to linear programming, to recognize the existence of quota filling assignments for a given instance of SMP and find one of them. If desirable, we can further efficiently transform $\xmax$ into $\xmin$, by using ``reversed'' rotations in $G$.

It should be noted, however, that when the constructed $y$ is not identically zero on $A\cup B$, we cannot guarantee that the restriction $x=y\rest{E}$ is stable for $G$. In light of this, one can address the question of constructing an efficient ``combinatorial'' method to find a stable assignment in all cases of SMP (when deficit vertices are possible).

\end{document}